\documentclass[a4paper,12pt, leqno]{article}
\usepackage{amsmath, amsthm, amssymb, fullpage}
\usepackage[english]{babel}
\usepackage[utf8]{inputenc}
\usepackage{eucal}
\usepackage{thmtools}
\usepackage{float} 
\usepackage{thm-restate}
\usepackage{tikz-cd}

\usepackage[colorlinks, bookmarks=true]{hyperref}
\hypersetup{
	colorlinks,
	citecolor=blue,
	linkcolor=blue,
	urlcolor=blue}
\usepackage{cleveref}
\usepackage{titlesec}
\usepackage{xcolor}
\usepackage{mathtools}
\usepackage[nottoc]{tocbibind}
\usepackage[noadjust]{cite}
\usepackage{graphicx}
\usepackage{caption}
\usepackage{subcaption}
\usepackage{bm}
\usepackage{mathrsfs}
\usepackage{comment}
\numberwithin{equation}{section}
\newcommand\numberthis{\refstepcounter{equation}\tag{\theequation}}

\newtheorem{theorem}{Theorem}[section]
\newtheorem{cor}[theorem]{Corollary}
\newtheorem{lem}[theorem]{Lemma}
\newtheorem{exa}[theorem]{Example}
\newtheorem{q}[theorem]{Question}
\newtheorem{prop}[theorem]{Proposition}
\newtheorem{rem}[theorem]{Remark}
\newtheorem{defi}[theorem]{Definition}

\newcommand{\di}{\,d}

\newcommand{\dau}{Daskalopoulos and Uhlenbeck }
\newcommand{\daun}{Daskalopoulos and Uhlenbeck}

\newcommand{\msl}{maximal stretch lamination }

\newcommand{\arccosh}{\textnormal{arccosh }}
\newcommand{\area}{\textnormal{Area}} 
\newcommand{\ehl}{$\epsilon$ homologically thick }

\newcommand{\bd}{\overline{D}}
\newcommand{\ws}{\widetilde{S}}
\newcommand{\br}{\overline{R}}
\newcommand{\bo}{\mathcal{O}}

\newcommand{\ds}{\displaystyle}

\newcommand{\entropy}{\textnormal{ent}}
\newcommand{\entphi}{\textnormal{ent}_{\phi}}

\newcommand{\fancyS}{\mathcal{S}}

\newcommand{\hy}{\mathbb{H}^3}
\newcommand{\inj}{\textnormal{inj}}
\newcommand{\ig}{I_{\gamma}}

\newcommand{\li}{L^{\infty}}
\newcommand{\lfourteen}{L14n21792 }

\newcommand{\flm}{Farre, Landesberg, and Minsky }
\newcommand{\flmn}{Farre, Landesberg, and Minsky}
\newcommand{\lip}{\textnormal{Lip}}

\newcommand{\mphi}{M_{\phi}}

\newcommand{\nz}{\mathbb{N}_0}
\newcommand{\wm}{\widetilde{M}}
\newcommand{\Q}{\mathbb{Q}}
\newcommand{\R}{\mathbb{R}}

\newcommand{\tn}{\textnormal}

\newcommand{\pc}{Poincar\'e }
\newcommand{\pig}{\Psi_{\ig}}
\newcommand{\zcover}{$\mathbb{Z}$-cover }
\newcommand{\skm}{s_{km}}
\newcommand{\T}{\mathbb{T}}

\newcommand{\da}{d_{\alpha}}

\newcommand{\vol}{\textnormal{Vol}}
\newcommand{\Z}{\mathbb{Z}}
\newcommand{\tfo}{{}_2F_1}

\newcommand{\PML}{$\mathbb{P}\mathcal{ML}$}

\newcommand{\stheta}{S_{\theta}}
\newcommand{\tj}{Thurston-J\o rgensen }
\newcommand{\diam}{\tn{diam}}

\title{Thurston norms, $L^2$-norms, geodesic laminations, and Lipschitz maps}
\author{Xiaolong Hans Han}
\date{}
\DeclareUnicodeCharacter{2212}{-}
\begin{document}
	\maketitle
	\begin{abstract}
		For closed hyperbolic $3$-manifolds $M$ with volume less than a constant $V$, we prove an inequality regarding the geometric $L^2$-norm and the topological Thurston norm, which is qualitatively sharp and verifies a conjecture of Brock and Dunfield in this case. Generically, we show that the $L^2$-norm is less than a constant $c(V)$ times the Thurston norm by showing that any least area closed surface is disjoint from the thin part. \par
		We then study the connection between the Thurston norm, best Lipschitz circle-valued maps, and maximal stretch laminations, building on the recent work of \daun \ and \flmn. We show that the distance between a level set and its translation is the reciprocal of the Lipschitz constant, bounded by the topological entropy of the pseudo-Anosov monodromy if $M$ fibers. For infinitely many examples constructed by Rudd, we show the entropy is bounded from below by one-third the length of the circumference.  
	\end{abstract}

\section{Introduction}
\indent	By Mostow's rigidity theorem, the geometric invariants of a closed hyperbolic $3$-manifold $M$ are determined by the topology of $M$. Effective geometrization attempts to clarify the quantitative relationship between the geometric and topological invariants.  There are two norms on the cohomology $H^1(M; \R)$, the topological Thurston norm and the geometric $L^2$-norm. For more background and history on the relationship between the two norms, we refer to \cite{hansHarmonicForms}. Brock and Dunfield \cite{bdNorms} prove the following inequality: 
	\begin{equation}\label{bdi}
		\ds \frac{\pi}{\sqrt{\vol(M)}} \| \cdot \|_{Th} \leq \| \cdot \|_{L^2} \leq \frac{10\pi}{\sqrt{\inj(M)}} \|\cdot \|_{Th} \text{ on } H^1(M;\R)
	\end{equation}
	where $\inj(M)$ is the injectivity radius of $M$.
 They construct a sequence of closed hyperbolic $3$-manifolds $M_j$ with uniformly bounded volume (from Dehn filling on a link complement), and prove in [\citenum{bdNorms}, Theorem 1.4] that
	\begin{equation}\label{eq conjecture}
			\|\cdot\|_{L^2} / \|\cdot\|_{Th} \sim	\bo\left(\sqrt{-\log\inj(M_j) }\right) \tn{ as } \inj(M_j) \rightarrow 0.
	\end{equation}
	They conjecture that this asymptotics reflects the leading order. In proving the global inequality 
	\begin{equation}\label{eq l2 less than root inj thurston}
		\|\cdot\|_{L^2} \leq \frac{10\pi}{\sqrt{\inj(M)}}\|\cdot\|_{Th} \tn{ on } H^1(M; \R),
	\end{equation}
 they first establish a local one of the \textit{same order} for harmonic $1$-forms
	\begin{equation}\label{eq l infinity bounded by 5 / square root inj times L2}
			\|\cdot\|_{\li} \leq \frac{5}{\sqrt{\inj(M)}}\|\cdot\|_{L^2},
	\end{equation}
	and then use tools from topology and geometry, such as the least area norm and the \pc duality, etc. By solving explicitly the Laplacian equation on a Margulis tube with twists in terms of hypergeometric functions with complex coefficients in \Cref{theorem harmonic expansion into hypergeometric}, we show in \Cref{local conjecture l infinity not bounded by sqrt inj l2} that the local version 
	\[\|\cdot\|_{\li} \leq c\sqrt{-\log \inj(M)}\|\cdot\|_{L^2}\] 
	of \eqref{eq conjecture} \textbf{does not hold} for harmonic $1$-forms on $M$, which provide potential counterexamples to their conjecture without the uniform volume assumption. Moreover, there exist Margulis tubes of arbitrarily large volume, so that [\citenum{bdNorms}, Theorem 1.4(b)] no longer holds\footnote{Specifically, the proof of (iii) on [\citenum{bdNorms}, p. 551] depends crucially on volumes of Margulis tube $V_n$ uniformly bounded above, since their $R_n$ using [\citenum{bdNorms}, (6.2)] depends on volume.}. Nevertheless, with several new tools, we establish their conjecture for manifolds with bounded volume. 		
	\begin{restatable}{theorem}{theoremMainConjectureBD}\label{theorem main conjecture BD}
		Let $V>0$. Then for all closed hyperbolic $3$-manifolds $M$ whose volume $\leq V$ and $\inj(M)\leq 58/10^6$, there exists a constant $c(V)$ such that 
		\begin{equation}\label{local conjecture}
			\ds \| \cdot \|_{L^2} \leq  \| \cdot \|_{Th} \left(2\sqrt{c^2(V)\pi^2 + 100\pi \log \left(\frac{V}{\pi\inj(M)}+1\right)} \right)	\tn{ on } H^1(M;\R). 
		\end{equation}
	\end{restatable}
\noindent	Note that by the \tj theory, there are only finitely hyperbolic manifolds with injectivity radius $\geq 58/10^6$ and volume $\leq V$. Such manifolds admit a uniform upper bound by \eqref{eq l2 less than root inj thurston}. Actually,  
	we show in \Cref{generic L2 thurston norm linear} that for a generic manifold $M$ of volume $\leq V$, the $L^2$-norm is bounded from above (and below) by some constant multiple of the Thurston norm, no matter how small the injectivity radius is: 
	$$  \| \cdot \|_{L^2} \leq  2c(V)\pi \| \cdot \|_{Th}   	\tn{ on } H^1(M;\R).$$
	One tool for establishing \Cref{theorem main conjecture BD} is a \textit{qualitatively sharp} geometric lower bound on the Thurston norm, whose proof is based on the interaction of short geodesics and minimal surfaces. 
		\begin{restatable}{lem}{thurstonNormLowerbound}\label{thurston norm lower bound by kappa / square root of length}
		If $M$ has $n$ geodesic $\gamma_i$ of length $\lambda_i\leq 0.01$, then the Thurston norm of $\alpha\in H^1(M;\R)$ satisfies
		\begin{equation}\label{eq lower bound on Thurston norm sum}
			\| \alpha \|_{Th} \geq \sum\limits_{i=1}^n \left|\int_{\gamma_i} \alpha\right| \left(\frac{0.107}{\sqrt{\lambda_i}}-1 \right). 
		\end{equation}
		In particular, if there are $m$ geodesics whose length $\lambda_i \leq 58/10^6$, then \begin{equation}\label{eq asymptotic lower bound on Thurston norm}
			\| \alpha \|_{Th} \geq\sum\limits_{i=1}^m \left|\int_{\gamma_i} \alpha\right| \frac{0.1}{\sqrt{\lambda_i}}.
		\end{equation}
	\end{restatable}
 	\Cref{thurston norm lower bound by kappa / square root of length} also simplifies the proof of [\citenum{bdNorms}, Theorem 1.4(b)] (see \Cref{application simplify Brock-Dunfield Theorem 1.4 (b)}) by considering the growth of the Thurston norm on some  cohomology classes of the Dehn fillings of $L14n21792$, which also shows that \Cref{thurston norm lower bound by kappa / square root of length} is qualitatively sharp. It admits a natural generalization to cusped manifolds (\Cref{thurstonNormLowerboundCusped}), based on the notion of cohomology slope, which computes the shortest loop on the boundary of a cusp neighborhood killed by a cohomology. 
 	
 	Another tool we develop concerns the boundary geometry of a Margulis tube and the volume upper bound $V$. In particular, using the idea of linear isoperimetric inequality and the foliation of a Margulis tube by flat tori, we bound the intrinsic diameter of the boundary of a Margulis tube in terms of $V$ in \Cref{uniform intrinsic injectivity radius and diameter}. We also show that a harmonic $1$-form on a tube admits an $L^2$-decomposition into topological terms and inessential terms in \Cref{harmonic 1-form decomposition in tube}, building on \eqref{harmonic expansion into hypergeometric}. Other tools developed include a series of results in \Cref{section preliminaries least area surfaces in Margulis tubes and the Thurston norm} on the interplay between minimal surfaces and harmonic $1$-forms on a tube. 
		
	In proving \Cref{theorem main conjecture BD}, we found that the Thurston norm is related to the Lipschitz geometry, which will have applications in bounding the fiber translation length of a pseudo-Anosov by its entropy (\Cref{separationBoundedAboveby entropy for a pseudo-Anosov}). To start with, the invariant $|\int_{\gamma}\alpha|/\sqrt{\ell(\gamma)}$, which appears in \eqref{eq lower bound on Thurston norm sum}, is a lower bound on the Thurston norm. If we modify the invariant 
	\[\frac{|\int_{\gamma}\alpha|}{\sqrt{\ell(\gamma)}} \hspace{0.15in}\tn{   into   } \hspace{0.15in} \frac{|\int_{\gamma}\alpha|}{ \ell(\gamma)}\eqqcolon K(\gamma),\]
	and take the supreme over the set $\fancyS$ of all free homotopy classes of simple closed curves in a closed hyperbolic $n$-manifold $M^n$ ($n\geq 2$), then
	\begin{equation}\label{K the geometric invariant sup intersection per length}
		K \coloneqq \sup\limits_{\gamma \in \fancyS}K(\gamma)
	\end{equation} 
	 coincides with the Lipschitz constant $\lip_{[\alpha]}$ of the cohomology class $[\alpha]$ considered by Daskalopoulos and Uhlenbeck  [\citenum{dubestLipLG}, Theorem 5.8]. For a map $f:M\rightarrow S^1$, we say $f\in \alpha$ if $\alpha=[f^*d\theta]$. It turns out $\lip_{[\alpha]}$ is equal to the smallest Lipschitz constant of circle-valued maps in the homotopy class $\{f\in\alpha\}$ and the norm $\|[\alpha]\|_{L^{\infty}}$ on the cohomology (see \Cref{def Lipschitz constant of cohomology} and the comments after that). Inspired by using \eqref{eq l infinity bounded by 5 / square root inj times L2} and least area surfaces to establish \eqref{eq l2 less than root inj thurston}, we prove 
	 \begin{restatable}{lem}{thurstonnormboundedbyLipschitzconstant}\label{thurston norm bounded by Lipschitz constant}
	 	Let $M$ be a closed hyperbolic $3$-manifold and $\alpha \in H^1(M; \Z)$. Then we have
	 	\begin{equation}\label{eq thurston norm bounded by Lipschitz constant}
	 		\|\alpha\|_{Th}< \frac{1}{\pi}\vol(M)\lip_{[\alpha]}. 
	 	\end{equation}
	 \end{restatable} 
\noindent	 Compared to \Cref{thurston norm lower bound by kappa / square root of length} which relates the Thurston norm to the length of closed geodesics, the best Lipschitz constant is realized at geodesic laminations, which are the same set of closed geodesics as in \Cref{thurston norm lower bound by kappa / square root of length} for infinitely many examples by [\citenum{rcStretchLamination}, Theorem 1.4]. This lemma is used to prove \Cref{inequality Lipschitz constant entropy} which shows that the Lipschitz constant (and sometimes the length of certain closed geodesics) of a fibered hyperbolic $3$-manifold is bounded by the entropy of its pseudo-Anosov monodromy. 

	 For $f \in\alpha$, we adapt the notion of a \textbf{tight} map  from \cite{flmMinimizingLaminationsRegularCover} whose Lipschitz constant $\lip_{f}$ is equal to $K$ in \eqref{K the geometric invariant sup intersection per length}, which necessarily minimizes the Lipschitz constant in its homotopy class $[f]$ (such maps are called best Lipschitz in \cite{dubestLipLG}). The maximal stretch locus of a general map is a subset of $M$ with maximal pointwise Lipschitz constant, and a \msl of a tight map is a subset of the maximal stretch locus which is a geodesic lamination (by [\citenum{flmMinimizingLaminationsRegularCover}, Proposition 3.1] the subset whose pullback to the frame bundle is invariant under the geodesic flow). Both notions are defined in \Cref{maximal stretch locus and lamination}. \par 
	 	Thurston \cite{thurstonNormHomology3manifold} uses the level sets of circle-valued maps to show an integral cohomology $\alpha \in H^1(M^3;\Z)$ is dual to embedded surfaces. Here, we establish some properties of the level sets of a tight circle-valued map. The \zcover $\widetilde{M}$ of $M$ associated with $\alpha$ and $f$ is an important tool for studying a tight map and its maximal stretch laminations. This is used in \cite{flmMinimizingLaminationsRegularCover, flmClassificationHorocycle} for classifying the horocycle orbit closures. We combine these two points of view and prove several results. The first one is \Cref{level set of tight cut lamination equally} which shows that the \msl $\Lambda(f)$ of a tight map $f$ cut along a level set $f^{-1}(\theta)$ results in geodesic segments of \textit{equal} length $1/K$. Furthermore, this number motivates the definition of the fiber translation length associated with $\alpha$. Suppose the deck transformation on $\widetilde{M}$ is $\tau$ and we have the following commutative diagram. 
	 \[
	 \begin{tikzcd}
	 	\wm \arrow{r}{\widetilde{f}} \arrow{d}{\pi}  & \R \arrow{d}{\pi} \\
	 	M \arrow{r}{f} & S^1.
	 \end{tikzcd}
	 \]
	 Let $S=f^{-1} (\theta)\subset M$ be a level set (which is possibly disconnected) and $\pi^{-1}(S)= \bigcup\limits_{i=-\infty}^{\infty} \tau^i\ws $ where $\ws$ lies on a fundamental domain for $(\wm, \langle\tau\rangle)$. In \Cref{fiber translation length of monodromy phi}, we define the fiber translation length $d_{\alpha}$, which is equal to
	 \[	  \sup_{\substack{S\subset M \tn{ isotopic} \\ \tn{to a fiber}}}	d_{\wm}(\widetilde{S}, \tau \cdot \widetilde{S}).\]  
	 Since the Lipschitz constant 
	 $$  \lip_{[f]}= \min\limits_{h\in [f]} \max\limits_{x\in M} \lip_h(x)$$ 
	 of a homotopy class is defined via min-max and is equal to $K$, we have
	 \begin{restatable}{theorem}{fibertranslationLengthequalReciprocalK}\label{fiber translation Length equal Reciprocal K}	 	
	 		The fiber translation length of $\alpha \in H^1(M^n;\Z)$ is equal to $\frac{1}{K}$
	 		$$
	 			\da=\frac{1}{K}. 
	 		$$
	 	 \end{restatable}
	 \begin{rem}
	 	 Since the \msl $\Lambda(f)$ maximizes the pointwise Lipschitz constant, in the \zcover $\wm$, $\widetilde{\Lambda(f)}$ goes through level sets from one end to the other as fast as possible\footnote{This principle is mentioned in \cite{flmMinimizingLaminationsRegularCover}.}. \Cref{fiber translation Length equal Reciprocal K} suggests that a tight map $f$ is not only tight in the gradient direction, but also has ``tight" level sets: non tight maps have level sets which are closer.
	 \end{rem}
\noindent	  With the fastest speed interpretation on $\widetilde{\Lambda(f)}$, \Cref{algebraic intersection=geometric intersection} shows that there exist natural orientations such that every algebraic intersection between $\Lambda(f)$ and some submanifold $\beta$ dual to $\alpha$ is positive. Moreover in $2$-dimension there exists a sublamination $\Lambda_m (\alpha) \subset \Lambda(f)$ which carries a transverse measure. We show that the ``algebraic intersection number" $\int_{\Lambda_m(f)} \alpha$ is equal to the geometric intersection number and the transverse measure. With a suitable normalization $\ell(\Lambda_m (\alpha))=1$, it suggests that  in a hyperbolic $n$-manifold we may define the geometric intersection of $\Lambda_m (\alpha)$ and $\beta$ as $K$, which generalizes the intersection form for a geodesic lamination and multi-curve on a hyperbolic surface and also the ordinary geometric intersection number if $\Lambda_m(\alpha)$ consists of closed geodesics. This is \Cref{defi geometric=K times length}. 
	 
 \begin{rem}
If the maximal stretch lamination is a union of closed geodesics $\gamma$ which intersects the genus-$g$ fiber $\kappa$ times, then \eqref{eq thurston norm bounded by Lipschitz constant} gives a quick lower bound $\frac{2\pi(g-1)|\kappa|}{\ell(\gamma)}$ for the volume of $M$. Thurston in [\citenum{wtMinimalStretchMaps}, Theorem 10.7] has shown the \msl for a generic surface homeomorphism is a simple closed geodesic. 
 \end{rem}
 \Cref{thurston norm bounded by Lipschitz constant} that relates the Thurston norm to the Lipschitz constant also admits an application in comparing the volume of a hyperbolic mapping torus with a Seifert-fibered mapping torus. 
\begin{restatable}{cor}{volTrivialMappingtorusLessthanHyperbolic}\label{vol Trivial Mapping torus Less than Hyperbolic}
 Let $\phi:S_g \rightarrow S_g$ be a pseudo-Anosov, $M_{\phi}$ the hyperbolic mapping torus, and $f: M_{\phi} \rightarrow S^1$ a fibration with fiber $S_g$. Equip $S_g$ with a hyperbolic metric and let $M_{Id}=S_g\times S^1$ be a product where $S^1=\R/\lip_{[f]}\Z$. Then we have 
$$
	\frac{1}{2}\vol(M_{Id}) < \vol(M_{\phi}).
$$
\end{restatable}
\noindent The $\frac{1}{2}$ coefficient boils down to an observation of Uhlenbeck, which restricts the area of a stable minimal surface $S_g$ in a hyperbolic $3$-manifold: $\frac{1}{2}\leq \area(S_g)/2\pi(g-1) \leq 1$. \par
\Cref{K the geometric invariant sup intersection per length} $\lip_{[\alpha]} = \sup_{\fancyS}\frac{|\int_{\gamma}\alpha| }{\ell(\gamma)}$ reveals one geometric aspect of the Lipschitz constant $\lip_{[\alpha]}$. We find out that it is also closely related to dynamical invariants such as the entropy of pseudo-Anosov monodromy when $\alpha$ is a fibered class\footnote{\dau in [\citenum{dubestLipLG}, Problem 9.10, 9.11] point out a potential connection between fibered hyperbolic $3$-manifolds $M$ and the infinity harmonic maps which are best Lipschitz.}. Suppose $f: M\rightarrow S^1$ is a fibration with pseudo-Anosov monodromy $\phi$. Combining  \Cref{thurston norm bounded by Lipschitz constant} and Kojima-McShane \cite{kmNormalizedEntropyVolumePA}, we establish an inequality relating the geometric Lipschitz constant $K$ to the topological entropy of $\phi$. 
\begin{restatable}{theorem}{inequalityLipschitzConstantEntropy}\label{inequality Lipschitz constant entropy}
	Let $\mphi$ be the mapping torus associated to $(S, \phi)$ with the induced fibration $f: M\rightarrow S^1$. Then the Lipschitz constant $K$ of $[f]$ and the entropy of $\phi$ satisfy 
$$
	K \cdot \tn{ent}_{\phi}>\frac{1}{3}.
$$
\end{restatable}
\noindent 
This constant lower bound is qualitatively sharp in view of 
\begin{restatable}{lem}{productLipschitzconstantentropy}\label{lem product Lipschitz entropy}
	Let $M$ be a hyperbolic mapping torus. Then there exists a sequence of finite covers $M_n$ with $f_n: M_n \rightarrow S^1$, $K_n$, and $\tn{ent}_n\rightarrow \infty$ so that 
	$$ K_n \cdot\tn{ent}_n = \tn{constant}.$$
\end{restatable}

\noindent As a corollary of the \Cref{inequality Lipschitz constant entropy,fiber translation Length equal Reciprocal K}, we have 
\begin{restatable}{cor}{separationBoundedAbovebyEntropy}\label{separationBoundedAboveby entropy for a pseudo-Anosov}
The fiber translation length $d_{\alpha}$ is bounded from above by the topological entropy of $\phi$:
	\begin{equation}\label{distance bounded by 3entropy}
		d_{\alpha} < 3\entphi.
	\end{equation}
\end{restatable}
[\citenum{ctGroupInvariantPeanocurve}, Theorem 5.1] shows that given a pseudo-Anosov $\phi: S\rightarrow S$ with associated unstable lamination $(\lambda_1, dx)$ and stable lamination $(\lambda_2, dy)$ and stretch factor $k>1$, there exists a natural pseudometric 
$$ds^2=k^{2t}dx^2+k^{-2t}dy^2+(\log k)^2dt^2$$
so that $ds^2$ and the hyperbolic metric on the mapping torus $M_{\phi}$ are quasicomparable. For a pseudo-Anosov, $\log k=\entphi$. The metric $(M_{\phi}, ds^2)$ is called singular Sol.
\begin{figure}[H]
	\includegraphics[scale=0.5]{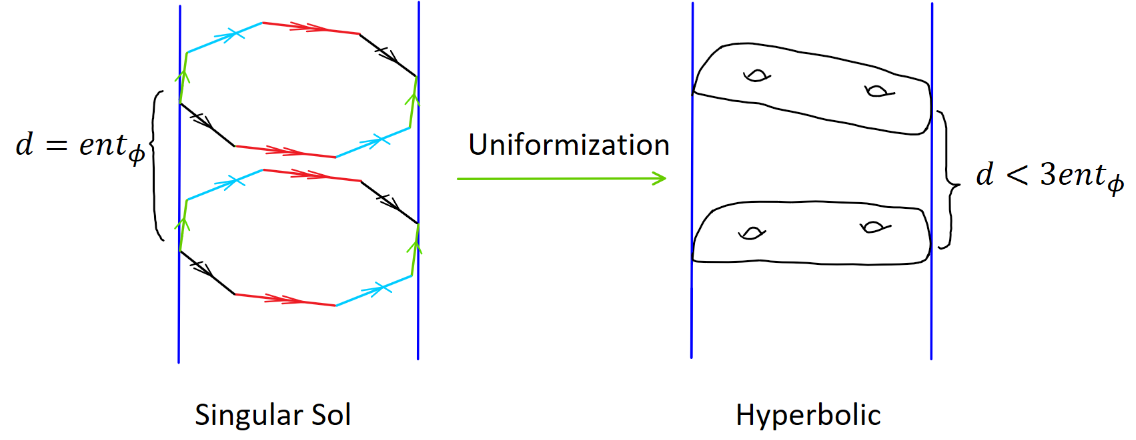}
\end{figure}
\noindent In their seminal works, Minsky \cite{myModelBoundsKleinianGroup}, and Brock, Canary, and Minsky \cite{bcmEndingLamination} establish a biLipschitz model and settle the ending lamination conjecture. See also \cite{bowditch2011ending} for other proofs of the conjecture. There are other translation lengths for pseudo-Anosov that are bounded by its entropy. 
Denote $\ell_{\mathscr{C}}(\phi)$ the asymptotic translation length of $\phi$ on the curve complex $\mathscr{C}$ of a genus-$g$ surface. Define $c_g$ as the optimal Lipschitz constant for the systole map $\text{sys}: \mathscr{T}(S) \rightarrow \mathscr{C}^{(1)}(S)$, defined by Masur and Minsky in \cite{mmGeometryComplexCurveHyperbolicity}. 
In [\citenum{ghklLipschitzconstantCurveComplex}, Lemma 3.2], Gadre, Hironaka, Kent, and Leininger prove
\[ \ell_{\mathscr{C}}(\phi) \leq c_g \entphi .\]
If we define the \textbf{circumference} $\gamma$ as the shortest loop that projects nontrivially to $\pi_1S^1$, then for $K(\gamma)=|\int_{\gamma}f^*d\theta |/ \ell(\gamma)$, a natural question is how $\ell(\gamma)$ compares with the entropy of $\phi$. By investigating the interaction between the thick-thin decomposition and the best lamination, Rudd in [\citenum{rcStretchLamination}, Theorem 1.4] constructs infinitely many fibered hyperbolic $3$-manifolds whose circumferences intersecting each fiber exactly once are the maximal stretch lamination for some fibered classes. As a corollary of \eqref{distance bounded by 3entropy} which bounds the entropy from below by the fiber translation length, we have
\begin{cor}\label{eq entropy larger than circumference}
	There exists a sequence of fibered hyperbolic $3$-manifolds $M_j$ which geometrically converge to the complement of the link \lfourteen such that the entropy $\entropy_{\phi_j}$ is bounded from below by the length of the circumference $\gamma_j$
	$$
		\entropy_{\phi_j} > \frac{1}{3}\ell(\gamma_j). 
	$$
\end{cor}
\noindent A natural question is whether this inequality always holds, which essentially asks about the structure of \msl for fibered cohomology classes. A related question is
\begin{q}
Does there exist a \msl in a closed hyperbolic $3$-manifold which is not a union of closed geodesics? 
\end{q}
Biringer and Souto in \cite{bsRanksMappingTori} show that the $\epsilon$-electric circumference is $\approx$ the translation length of $\phi$. Recently, Italiano, Martelli, and Migliorini \cite{imm-hyp5ManifoldFiberoverthecirclet} made a breakthrough by discovering a $5$-dimensional hyperbolic manifold that fibers over the circle and solving conjectures in geometric group theory. Even though \cite{kmNormalizedEntropyVolumePA} builds upon many tools that rely on three-dimension crucially, the Lipschitz constant, fiber translation length, and topological entropy are independent of dimension. It is natural to wonder whether this generalizes to higher dimensional fibered hyperbolic manifolds. 

The lower bound of $\frac{\|\cdot \|_{L^2}}{\|\cdot \|_{Th}}$ (i.e., the left-hand side of \eqref{bdi}) is investigated in [\citenum{hansHarmonicForms}, Section 5]. In \Cref{section preliminaries least area surfaces in Margulis tubes and the Thurston norm}, we provide the preliminaries on the least area surfaces and the Thurston norm. One theme is the intricate interplay between harmonic $1$-forms and least area surfaces in a tube. 
	In \Cref{section the proof}, we prove \Cref{theorem main conjecture BD}. In \Cref{section generic growth}, we develop the notion of ``generic Dehn filling" and \ehl and show that for a generic manifold $M$ with $\vol(M)\leq V$,  $\|\cdot\|_{L^2}/\|\cdot\|_{Th}$ is bounded by a constant $c(V)$. 
	In \Cref{section lipschitz constant and entropy}, we prove some basic properties of the maximal stretch laminations and the level sets of a tight circle-valued map. We then study the Lipschitz constant and its connection with the Thurston norm and the entropy of pseudo-Anosov monodromy.
	\subsection*{Acknowledgments}
The author would like to express deep gratitude to his advisor, Nathan Dunfield, and his postdoc mentor, Yunhui Wu. The author thanks Yitwah Cheung, James Farre, Yi Huang, Chris Leininger and Cameron Gates Rudd for helpful discussions on laminations. He also thanks Yilin Gong, Zeno Huang, and Han Hong for helpful discussions on minimal surfaces and harmonic forms. This work is partially supported by NSF grant DMS-1811156 and DMS-1928930 while the author participated in a program hosted by the Mathematical Sciences Research Institute in Berkeley, California, during the Fall 2020 semester.
	\tableofcontents	

\section{Preliminaries}\label{section preliminaries least area surfaces in Margulis tubes and the Thurston norm}
	We assume throughout this section that $M$ is a closed hyperbolic $3$-manifold with volume $<V$ and $b_1(M)>0$. 
We start by defining cylindrical coordinates on a Margulis tube $\T$. Denote the length of its core geodesic $\gamma$ by $\lambda$ and its radius by $R$. Let $z \in [0, \lambda]$ be a unit-speed parameterization of the geodesic. Let $r \in [0, R]$ be the radial coordinate and $\theta$ be the angular coordinate. The coordinates come with identification $(r, \theta, \lambda)\sim (r, \theta +\theta_0, 0)$ where $\theta_0$ is the twist parameter determined by the core geodesic of $\T$. The metric on $\T$ is then given by 
	\begin{equation}\label{eq metric Tube}
			g=dr^2+\sinh^2{r} \di \theta^2 + \cosh^2{r} \di z^2.
	\end{equation}	
	The metric on a totally geodesic disk $D=\{z=c\}$ is $dr^2+\sinh^2{r} d\theta^2$. The boundary of the Margulis tube $\partial \T_R=\partial \T_R$ has metric $\sinh^2 R \di\theta^2 + \cosh^2{R} \di z^2$ which is homogeneous and thus must be intrinsically flat by the Gauss–Bonnet theorem. It is also convex with respect to inward-pointing normal vectors. The meridian of $\partial \T=\partial \T$ is an isotopy class of curves which is contractible in $\T$. 
	 If $\epsilon=0.29$ is a Margulis constant for $M$, and $\lambda$ is the real length of the core geodesic of a Margulis tube of $M$, then by [\citenum{fpsEffectiveDistanceNestedTubes}, Theorem 1.1] the tube radius $R$ satisfies
	\begin{equation}\label{tube radius lower bound}
		\cosh R \geq \frac{0.29}{\sqrt{7.256\lambda}}. 
	\end{equation}
	For $\epsilon>0$, we define the $\epsilon$-thick part of $M$ as 
	$$M_{\geq \epsilon}=\{x\in M: \inj(x) \geq \frac{1}{2}\epsilon\}.$$
	The $\epsilon$-thin part is $M_{< \epsilon} = M-M_{\geq \epsilon}$. \\
	\begin{lem}\label{uniform intrinsic injectivity radius and diameter}
		Let $V>0$. Then for all closed hyperbolic $3$-manifolds $M$ whose volume $\leq V$ and $b_1(M)>0$ and any Margulis tube $\T \subset M$, we have 
		\begin{enumerate}
			\item the intrinsic injectivity radius of $T$ is $\geq 0.145$, 
			\item there exists a constant $d(V)$ such that the intrinsic diameter $\diam (T)$ of $T$ is bounded from above by $d(V)$. 
		\end{enumerate}
	\end{lem}
	\begin{proof}
		By [\citenum{csMargulis_number_Haken_manifolds}], $0.29$ is a Margulis constant for hyperbolic $3$-manifolds with a positive betti number. The intrinsic injectivity radius of $T = \partial M_{\leq 0.29}$ is $\geq 0.29/2$ since $T$ is a submanifold. Denote $\lambda$ the length of the core geodesic of $\T$. 
		Using \eqref{eq metric Tube}, the area of $\partial \T$ and the volume of $\T$ are
		\begin{equation}\label{eq area of tube}
			\ds \area(\partial \T)=\int_{0}^{\lambda}\int_{0}^{2\pi} \sinh R \di \theta \wedge \cosh R\,dz =2\pi\lambda\sinh R \cosh R. 
		\end{equation} 
		and 
		\begin{equation}	 \label{eq volume of tube}
			\ds \vol(\T)=\int_{0}^{\lambda}\int_{0}^{R}\int_{0}^{2\pi}dr \wedge \sinh r \di \theta \wedge \cosh r\,dz =\pi\lambda\sinh^2R.
		\end{equation}
		On the flat torus $\partial \T$, since the product of intrinsic injectivity radius and the diameter is $\leq $ area, we have 
		$$0.145 \cdot \diam (T) \leq 2\pi\lambda\sinh R \cosh R \leq 2\vol(\T)+ 2\sqrt{\pi\lambda} \sqrt{\vol(\T)} \leq 2V+2\sqrt{0.29\pi} \sqrt{V}.$$ 
	\end{proof}
	\begin{lem}\label{R tube injectivity radius}
		If a Margulis tube $\T_{R+1}$ has radius $R+1$ and Margulis constant $0.29$, then the boundary $T_R=\partial \T_R$ of the $R$-subtube has injectivity radius at least $ 0.00227$. 	
	\end{lem}
	\begin{proof}
		Denote $\delta$ the Margulis constant of the $R$-subtube $\T_R$. By the left-hand side of [\citenum{fpsEffectiveDistanceNestedTubes}, Theorem 1.1], we have
		$$\arccosh \frac{0.29}{\sqrt{7.256\delta}} -0.0424\leq 1$$
		which implies that 
		$$\delta \geq \frac{1}{7.256}\left(\frac{0.29}{\cosh^2 1.0424}\right)^2 \geq 0.00455986.$$
		Thus the injectivity radius $\geq \delta/2 \geq 0.00227993$. 
	\end{proof}
	By \eqref{eq metric Tube}, an orthonormal basis of $1$-forms on $\T$ is 
	\begin{equation}\label{tubeONB1Forms}
		\{dr, \sinh{r} d\theta, \cosh{r}dz\}.
	\end{equation}
	For a $1$-form $\alpha$, the $L^2$-norm of $\alpha$ is $\|\alpha\|_{L^2}=\sqrt{\int_M \alpha \wedge *\alpha}$. For a cohomology class $[\alpha]$, the $L^2$-norm is $\|[\alpha]\|_{L^2}\coloneqq\inf\limits_{\beta \sim \alpha}\|\beta\|_{L^2}$, which by the Hodge theory is realized when $\alpha$ is harmonic. We assume $\alpha$ is harmonic in \Cref{section preliminaries least area surfaces in Margulis tubes and the Thurston norm,section the proof,section generic growth} unless specified otherwise. 
	\subsection{Least area surfaces in a Margulis tube}
Let $\T$ be a Margulis tube. For a least area surface $S$, the intersection $S\cap \T$ consists of least area subsurfaces $S'$ with boundary on $\partial \T$. Since $\T$ is convex with respect to the inward-pointing normal vectors, the convex hull property [\citenum{cmCourseMinSurface}, Proposition 1.9] of minimal surfaces forces the interior of $S'$ to be disjoint from $\partial \T$. Thus we have all components of the intersection of a least area surface $S$ and a Margulis tube $\T$ are properly embedded in $\T$. 

	Topologically, the intersection of $S$ and $\T$ consists of essential disks, boundary-parallel disks, and incompressible and compressible annuli, which we first define for reference.
	\begin{defi}
		A properly embedded surface $S' \subset \T$ with $\partial S' \subset \partial \T $ is called \textbf{boundary-parallel} if $S'$ is homotopic onto $\partial \T$ relative to its boundary. A surface $S'$ is \textbf{essential} if $[S', \partial S']\in H_2(\T, \partial \T; \Z)$ has a nonzero algebraic intersection with the core geodesic $\gamma$.
	\end{defi}
	
 The incompressibility of $S$ does not ensure the annuli components of $S\cap\T$ are incompressible. We first rule out a separating and compressible annulus from the intersection of a (not necessarily incompressible) minimal surface $S$ and the tube $\T$. Farre and Vargas Pallete prove that a minimal annulus in $\T$ that is a subset of an incompressible surface is incompressible in the last two paragraphs of their proof of [\citenum{fpMinimal-area-surfaces-fibered-3-manifolds}, Proposition 2.1]. 
	\begin{lem}
		Let $\T$ be a Margulis tube and $S$ a minimal surface. Then the intersection of $S$ and $\T$ cannot contain  a separating and compressible annulus component. In particular, an annulus component of the intersection of a minimal incompressible surface and the tube is incompressible and parallel to $\partial \T$.  
	\end{lem}
	\begin{proof}
		 Suppose a minimal surface $S$ intersects $\T$ into a separating and compressible annulus $A$ with boundary curves $\gamma_1$ and $\gamma_2$. Then one of the three cases occur:
		 \begin{enumerate}
		 	\item $\gamma_1$ and $\gamma_2$ are parallel to the meridian of $\partial \T$,
		 	\item $\gamma_1$ and $\gamma_2$ both bound disks on $\partial \T$, or
		 	\item  $\gamma_1$ is parallel to the meridian and $\gamma_2$ bound a disk on $\partial \T$.
		 \end{enumerate}
		    In all three cases, the $\gamma_i$ are homotopically trivial in $\T$. Since $A$ is separating, the surface $S$ is decomposed into $ D_1 \cup A \cup D_2$, such that $D_i \cap A =\gamma_i$. The union of the subsurface $D_2$ and $A$ forms a minimal surface $S_2$ whose boundary is $\gamma_1$. This is a contradiction since the boundary $\partial \T$ of Margulis tube $\T$ is convex with respect to inward-pointing normal vectors. Thus the minimal surface $S_2$ whose boundary $\partial S_2 \subset \partial \T$ cannot intersect or touch $\partial \T$ by the convex hull property of minimal surfaces [\citenum{cmCourseMinSurface}, Proposition 1.9]. 
		 
		If $S$ is incompressible and $\gamma_i$ are homotopically trivial, either $\gamma_1$ or $\gamma_2$ must bound a disk $D \subset S$, say, $ D \cap A = \gamma_2$. Then $S_2 \coloneqq D \cup A$ is a minimal surface whose boundary is $\gamma_1$. By the same convexity argument, the intersection $S_2 \cap \partial \T$ consists of only  $\gamma_1$ and cannot contain $\gamma_2$. Thus the annulus $A=S \cap \T$ is incompressible and by [\citenum{mbIntroGeometricTopology}, Proposition 9.3.16] $\partial$-parallel. 
	\end{proof}
Since an incompressible annulus intersects every essential disk in $\T$, we have the following duality for embedded least area surfaces.
\begin{lem}\label{classification surfaces in tube}
	Topologically, two cases exist for the intersection of an embedded least area surface $S$ and a Margulis tube $\T$. It consists of either 
	\begin{enumerate}
		\item essential disks,  or
		\item incompressible annuli,
	\end{enumerate}
with probably some boundary-parallel disks. 
\end{lem}
\noindent Huang and Wang in [\citenum{hwCSC}, Section 3.2] show that a minimal incompressible annulus is lifted to a minimal hyperbolic helicoid in $\hy$. 

Now we show that regardless of the area of the surface $S \subset M$, a boundary-parallel disk component of $S \cap \T$ stays uniformly close to the boundary of $\T$. The essential condition is that the boundary of the Margulis tubes have uniformly bounded diameters. 
\begin{lem}\label{uniform-depth-boundary-parallel-disks}
	Let $\T$ be a Margulis tube with radius $R$. Suppose the intersection of a least area $S$ and $\T$ contains a component which is a boundary-parallel disk $B$. Then there exists a uniform constant $\br(V)$, called the \textbf{depth}, such that $B$ is contained in the $\br$-neighborhood of $\partial\T$: 
	\[ B \subset N_{\br}(\partial\T) .\] 
	In particular, there exists a uniform constant $\epsilon'$ such that at any point $x \in B'$, 
	\[ \inj_x \geq \epsilon'(V). \] 
\end{lem}
\begin{proof}
	A least area surface $S$ is stable, and by the classic result of Schoen \cite{srEstimatesStableMinimalSurface}, as a subset of hyperbolic manifolds it has a uniform curvature lower bound. It is a consequence of the Gauss equation that a minimal surface $S$ in a hyperbolic $3$-manifold has a uniform curvature upper bound. The boundary torus $T$ of the Margulis tubes $\T$ has uniform diameter upper bound $d(V)$. Thus we can apply [\citen{mrMinSurfaceShort}, Corollary 7] to conclude $B \cap N_{R-\br}(\gamma) = \emptyset$ for some constant $\br(V)$.  Since $T$ has injectivity radius $\epsilon/2$, and $\partial B_{\br}(T)$ is uniformly away from $T$, by [\citenum{fpsEffectiveDistanceNestedTubes}, Theorem 1.1] any point $x \in B_{\br}(T)$ has a uniform injectivity radius lower bound $\epsilon'= \frac{1}{2}\frac{\epsilon^2}{7.256\cosh^2 (\br+0.424)}$. 
\end{proof}
In [\citenum{fpMinimal-area-surfaces-fibered-3-manifolds}, Proposition 2.1], Farre and Vargas Pallete control the depth of a sequence of least area $S_j$ with uniform area upper bound via a lower bound on $\theta_0^2 / \lambda$, which is satisfied by a sequence of manifolds $M_j$ which geometrically converge ([\citenum{fpMinimal-area-surfaces-fibered-3-manifolds}, Theorem 2.5]). The $M_j$ necessarily has a uniform volume upper bound so that the above lemma applies. Their argument relies on concrete hyperbolic geometry instead. 

\begin{defi}
	An \textbf{invariant} annulus $A\subset \T$ is an annulus which is invariant under the $1$-parameter family $\Psi_{\ig}(t)$ of isometry generated by the loxodromic isometry $\ig$. 
\end{defi}
For example, the minimal annuli covered by hyperbolic helicoid in \cite{hwCSC} are invariant. Let $f$ be a harmonic function defined in $\T$ (whose harmonic expansion is given by \Cref{theorem harmonic expansion into hypergeometric}).
\begin{lem}\label{int-*df-0-invariant-annulus}
	Let $A\subset \T$ be an invariant annulus. Then 
	$$
		\ds \int_A *df =0. 
	$$
	In particular, for any $R>0$, 
	\[\ds \int_{\partial \T_R} *df =0.\]
\end{lem}
\begin{proof}
	The annulus $A$ is homotopic to an annulus $A_R \subset \partial \T_R$, fixing the boundary $\partial A \subset \partial \T_R$. Since $\partial A$ and $\partial \T_R$ are invariant under $\Psi_{\ig}(t)$, $A_R$ is invariant as well. Thus there exists a parameter $\eta$ that measures the ``angle spread" of the annulus $A_R$ so that $A_R$ can be parameterized as 
	\[A_R=\{(R, \theta, z)| \theta \in [-\frac{z}{\lambda}\theta_0, \eta -\frac{z}{\lambda}\theta_0 ], z\in [0, \lambda]\}.\] 
	Recall from \eqref{skmThetaz} if $(k,m) \neq (0,0)$,
	\[\skm(\theta,z) =a_{km} \sin(k\theta+\frac{2\pi m}{\lambda}z+\frac{k\theta_0}{\lambda}z)+a_{km} '\cos(k\theta+\frac{2\pi m}{\lambda}z+\frac{k\theta_0}{\lambda}z)\] and $s_{00}=0$. 
	
Since an orthonormal basis of $1$-forms on $A_R$ is given by $\{\sinh R  \, d\theta, \cosh R \,dz\}$, we have
\begin{align*} 
	\int_{A_R} *df &= \int_{0}^{\lambda}\int_{-\frac{z}{\lambda}\theta_0}^{\eta -\frac{z}{\lambda}\theta_0}\sum_{k,m}^{\infty} f_R \sinh R  \, d\theta \wedge \cosh R \,dz\\
	&=  \sum_{k,m}^{\infty} h_{km}'(R) \sinh R \cosh R\int_{0}^{\lambda}\int_{-\frac{z}{\lambda}\theta_0}^{\eta -\frac{z}{\lambda}\theta_0} s_{km}(\theta,z) \, d\theta \wedge dz.
\end{align*}
For the integral in each summand with $k \neq 0$, 
\begin{align*}
	&\int_{0}^{\lambda}\int_{-\frac{z}{\lambda}\theta_0}^{\eta -\frac{z}{\lambda}\theta_0} \sin(k\theta+\frac{2\pi m}{\lambda}z+\frac{k\theta_0}{\lambda}z)d\theta \wedge dz\\ &= \frac{1}{k} \int_{0}^{\lambda} \cos(\frac{2\pi m}{\lambda}z)- \frac{1}{k} \int_{0}^{\lambda} \cos(k\eta+\frac{2\pi m}{\lambda}z) dz =0.
\end{align*}
If $k=0$, for $m\neq 0$,
\[\int_{0}^{\lambda} \sin\left(\frac{2\pi m}{\lambda}z \right)dz= \int_{0}^{\lambda} \cos\left(\frac{2\pi m}{\lambda}z \right)dz= 0.\]
which implies $	\int_{A_R} *df=0 $. Applying Stokes' theorem to the region bounded by $A$ and $A_R$, we obtain $\int_{A} *df=\int_{A_R} *df=0$. For the whole torus, take $\eta=2\pi$ and the conclusion follows as well. 
\end{proof}
 The Margulis tube $\T$ is foliated by totally geodesic disks $\bd \coloneqq \{z=c\}$. The non-transverse intersection of a least area disk $D$ and a single $\bd$ can be perturbed away by replacing $\bd$ with a nearby totally geodesic disk. 

\begin{lem}\label{D-transversely-intersect-geodesic-foliation-z=c}
	Suppose a totally geodesic disk $\bd=\{z=c\}$ has some non-transverse intersections with a least area disk $D$. Then there exists a $\delta_1$ such that $\bd_{\delta_1}=\{z=c + \delta_1\}$ only has transverse intersections with $D$. 
\end{lem}
\begin{proof}
	By [\citenum{hj-minimal-surfaces-foliated-manifolds}, Lemma 2.3] the number of non-transverse intersections between $D$ and the minimal foliation by totally geodesic disks $\bd=\{z=constant\}$ is finite and isolated. Thus there exists a $\delta_1$ such that $\bd_{\delta_1}$ is void of non-transverse intersections with $\bd$. 
\end{proof}
We will henceforward assume that the least area disk $D \subset S$ only has transverse intersections with a fixed totally geodesic disk $\bd$.
\begin{lem}
The transverse intersection	between a least area disk $D$ and $\bd$ consists of $k$ embedded, mutually disjoint arcs $e_1, \cdots, e_k$, all of whose endpoints lie on $\partial D \subset \partial \T_R$.
\end{lem}
\begin{proof}
	The intersection of a least area disk $D$ and a totally geodesic $\bd$ cannot contain a closed loop $\beta$ component. Otherwise, $\beta$ bounds a minimal disk $D' \subset D$, which by the convex hull property $D'$ is a subset of the convex hull of $\beta \subset \bd$. Moreover, it cannot contain an arc with either endpoint lying in the interior of $D$ by [\citenum{bwMinSurfaceLecNote}, Corollary 4.4]. Since the least area disk $D$ and $\bd$ are embedded, their intersection consists of embedded arcs. 
\end{proof}
	Order them naturally by $e_i$ so that the endpoints of $e_i$ separate the endpoints of $e_1, \cdots, e_{i-1}$ and $e_{i+1}, \cdots, e_k$ on $\partial \bd$.  The two outermost edges are $e_1$ and $e_k$ that trap all the other $k-2$ arcs. The arcs divide the disk $\bd$ into $\bd_1, \cdots, \bd_{k+1}$. Consider a fundamental domain $\widetilde{\T}$ of $\T$ in $\hy$, whose boundary consists of two totally geodesic disks (which, for the simplicity of notations, we denote by) $\bd$ and $\bd'$, and also $\partial \T_R$. The arcs divide the disk $D$ into $D_1, \cdots, D_{k+1}$ and similarly, there are $k$ arcs intersection $e_1', \cdots, e_k' \subset \bd'$.
	\begin{figure}[H]
	\centering
		\subfloat[Edges $e_i$ of $\bd \cap D$]{
		\includegraphics[width=.3\linewidth]{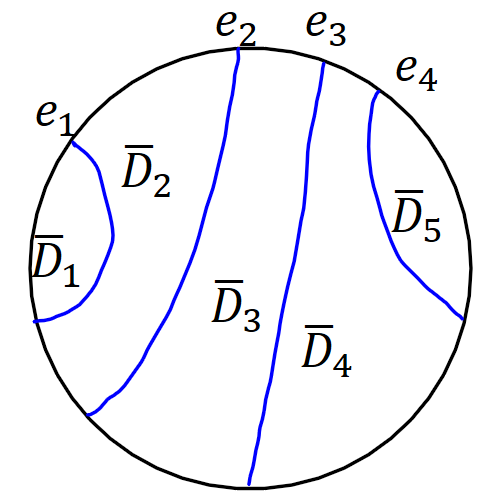}
	}\hspace{0.5in}
		\subfloat[Least area $D$ intersecting the fundamental domain $\widetilde{\T}$]{
		\includegraphics[width=.3\linewidth]{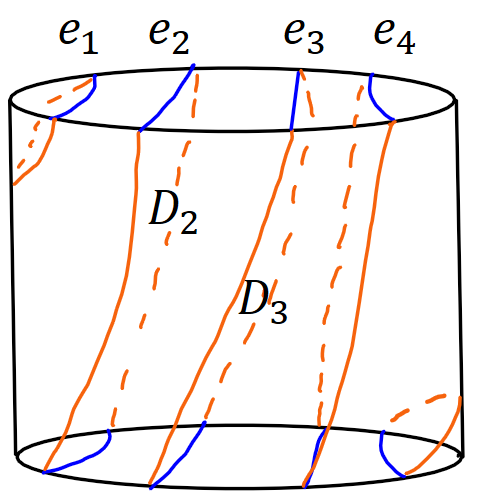}
	}
	\end{figure}
	 We need a further classification on the transverse intersection of $D$ and $\bd$ which captures a certain \textit{complexity} of $D$ in $\T$. 
	 \begin{defi}
	 	We call an arc $e_k$ of the transverse intersection between $D$ and $\bd$ \textbf{persistent} if in $\widetilde{\T}$, the boundary of $D_k$ contains $e_k\subset \bd$ and $e'_{k'}\subset \bd'$. 
	 \end{defi}
\noindent	 Intuitively, the least area $D_k$ intersects the top and bottom geodesic disk of the fundamental domain into $e_k$ and $e'_{k'}$, respectively.
	 The persistent intersection comes from a least area disk component $D_k$ which intersects \textit{every} totally geodesic leaf of the foliation $z=c$. The two-dimensional analogy is the following. We use an annulus embedded in the plane enclosed by two concentric circles $c_1$ and $c_2$ as a model for a tube. A segment in the annulus perpendicular to both $c_1$ and $c_2$ represents a totally geodesic disk. If we apply a Dehn twist to the segment around the core circle, we obtain a segment homotopic to the original one (via a homotopy that slides along the boundary circles of the annulus), but now it intersects each geodesic disk once. With a similar construction, the least area disk can intersect a given totally geodesic disk many times. \par
  	Unlike $*dz$ which is pointwisely zero on $\partial \T_R$, $*df = *f_r\di r$ is nonzero along $\partial \T_R$ in general (see \Cref{counter example Brock-Dunfield without GH convergence}). If the least area disk $D$ intersects $\bd$ many times, the homotopy which slides $D$ onto $\bd$ may sweep the boundary torus $\T$ many times. However, regardless of the number of persistent intersections, we can replace a least area disk $D$ with a totally geodesic disk with uniformly bounded small error. 
\begin{theorem}
	\[\int_D *df  \leq \sup\limits_{  \substack{\bd \tn{ totally} \\ \tn{geodesic disk}}} \int_{\bd} *df + \area(\partial \T_{R})\|f_r\di r\|_{L^{\infty}(\partial \T_{R})} . \]
\end{theorem}
\begin{proof}
If in $\T$ there exists a totally geodesic disk $\bd$ which is disjoint from $D$, then $\partial \bd, \partial D \subset \partial \T_R$ together bounds an annulus region $A \subset \partial \T_R$. By Stokes' theorem (with appropriate orientation on $A$) and $*df=*f_r\di r$ on $\partial \T_R$, we have
\[ \int_D *df =  \int_{\bd} *df + \int_A *df \leq  \int_{\bd} *df + \area(A) \|f_r\di r\|_{L^{\infty}(\partial \T_{R})}. \]
Now suppose the least area $D$ intersects every totally geodesic disk $\bd$. By \Cref{D-transversely-intersect-geodesic-foliation-z=c}, there exists a $\bd$ which intersects $D$ transversely. Via a homotopy that fixes the endpoints of $e_1$, we slide $D_1$ onto $\bd_1$ along a region on the boundary $W_1 \subset \partial \T_R$. Denote the endpoints of $e_i$ by $x_i$ and $y_i$ and correspondingly $x_i'$ and $y_i'$ of $e_i'$. The isometry $\ig$ corresponding to the core geodesic $\tilde{\gamma}$ in $\hy$ preserves the $r$-boundary torus $\partial \T_r$ and maps $x_i$ and $y_i$ to $x_i'$ and $y_i'$, respectively. The two points $x_i= (R, 0, z_1+\lambda)$ and $x_i'= (R, \theta_0, z_1)$ bound a ``vertical" segment ${x_ix_i'}=\{(R, \theta_0-\frac{t}{\lambda}\theta_0, z_1+t)|t\in [0, \lambda]\} $ invariant under $\pig$. The four segments $e_i$, ${x_ix_i'}, e_i'$, and ${y_iy_i'}$ bounds an $\pig$-invariant annulus $A_i$. 
\begin{figure}[H]
	\centering
	\includegraphics[scale=0.4]{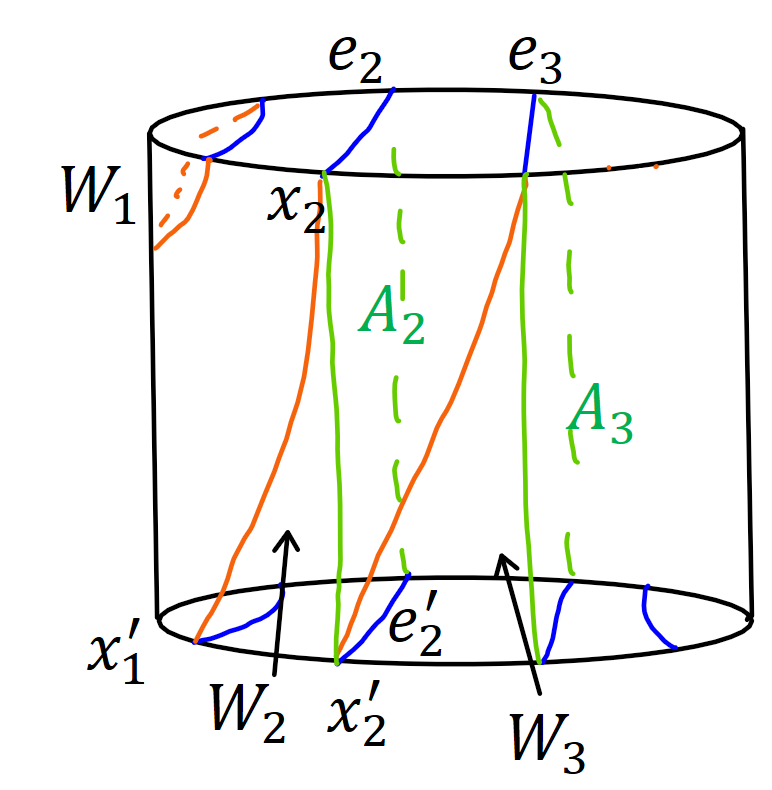}
	\caption{Glue invariant annulus $A_i$ to homotope $D_i$ to $\bd_i$}
\end{figure}
Using \eqref{int-*df-0-invariant-annulus}, we have \[\int_{A_i}*df=0\] for each $i$. The segments ${x_i'x_{i+1}'}, {x_ix_i'}$, and $ \partial D$ bound a region $W_i \subset \partial \T_R$ and ${y_i'y_{i+1}'}, {y_iy_{i}'}$, and $\partial D$ bound a region $W_i' \subset \partial \T_R$. Denote $W_i'' \coloneqq W_i \cup W_i'$. Then 
$$
	\int_{D_i} *df = \int_{A_i} *df + \int_{\bd_i} *df +  \int_{W_i''} *df =  \int_{\bd_i} *df +  \int_{W_i''} *df.
$$
Denote $W \coloneqq \cup W_i''$. Summing up, we have 
$$
		\int_{D} *df = \int_{\bd} *df +  \int_{W} *df.
$$
Since $D$ is embedded and the interior of $D$ is contained in $\T$, no distinct $W_i$ intersects in the interior. This implies that $W$ is an embedded subsurface of $T_R$. The conclusion then follows from $*df = *f_r\di r$ and  
$$\int_{W} *df \leq \area(W) \|f_r\di r\|_{L^{\infty}(\partial \T_{R})}.$$
\end{proof}
Since by \eqref{eq area of tube}
$\area(T_R) =\lambda 2\pi \sinh R \cosh R$
and by \eqref{eq volume of tube}
$\vol(\T_R) =\lambda\pi\sinh^2R,$
we have $\area(T_R) < 2V+2\sqrt{0.29\pi} \sqrt{V}$ by \Cref{uniform intrinsic injectivity radius and diameter}. 
\begin{cor}\label{integral *df least area close to totally geodesic}
Let $M$ be a hyperbolic $3$-manifold with volume $\leq V$ and $b_1(M)>0$. Then there exists $c_D(V)$ such that for all least area disk $D \subset S \cap \T$, there exists some totally geodesic disk $\bd$ with
	\[\int_D *df  \leq \int_{\bd} *df + (2V+2\sqrt{0.29\pi} \sqrt{V}) \|f_rdr\|_{L^{\infty}(\partial \T_R)}. \]
\end{cor}
For any harmonic $f$, \eqref{h'(r)< h(r)8d^2 logcoshr /sinh^2r} shows that $\|f_rdr\|_{L^{\infty}(\partial \T_R)} \rightarrow 0$ but $\int_{D_R} *df$ can tend to infinity as the radius $R$ of the totally geodesic disk $D_R$ tends to infinity. This shows that the error is small if we carry out the computation on a totally geodesic disk instead of a least area one. This corollary makes it possible to shed light on the conjecture without the volume upper bound, which we investigate in \Cref{counter example Brock-Dunfield without GH convergence}.

\subsection{The Thurston norm} 
\begin{defi}[The Thurston norm and taut surfaces]
	For a connected surface $S$, $\chi_-(S) \coloneqq \max\{-\chi(S),0\}$, where $\chi$ is the Euler characteristic. Extend this to disconnected surfaces via $\chi_-(S \sqcup S') =\chi_-(S) + \chi_-(S')$. For a compact irreducible $3$-manifold $M$, the \textbf{Thurston norm} of $\alpha \in H^1(M; \mathbb{Z})$ is defined by 
	\begin{center}
		$\| \alpha \|_{Th}= \min\{\chi_- (S) : S\text{ is an embedded surface dual to } \alpha \}$.
	\end{center}
	The Thurston norm extends to $H^1(M; \mathbb{Q})$ by making it linear on rays through the origin and then extends by continuity to $H^1(M; \R)$.
	A closed incompressible surface $S $ is called \textbf{taut} if $|\chi(S)|=\|[S]\|_{Th}$, i.e., it realizes the minimum topological complexity in its homology class. We assume $S\subset M$ is taut unless stated otherwise. 
\end{defi}
Let $M$ be a closed hyperbolic $3$-manifold, with $\alpha \in H^1(M; \Z)$ harmonic and dual to a least area $S \in H_2(M;\Z)$. Let $\T$ be a Margulis tube. Assume the algebraic intersection number between $S$ and $\T$ is $\kappa(\alpha)$ and the intersection contains essential disks. Since $H^1(\T)$ is $1$-dimension and $dz$ is a (harmonic) generator, on $\T$ the $1$-form $\alpha$ is cohomologous to $\frac{\kappa(\alpha)}{\lambda}dz$ (so that $\int_{\gamma} \alpha =\kappa(\alpha)$). The number $\kappa(\alpha)$ coincides with the winding number related to the Thurston norm on [\citenum{btDehnThurston}, p.392]. Recall from \Cref{classification surfaces in tube} that the components of $S\cap \T$ are essential and boundary-parallel disks. \par
The algebraic intersection number $\kappa=\kappa(\alpha)$ of $S$ and $\gamma$ and the length of $\gamma$ provide a lower bound on the Thurston norm:
\thurstonNormLowerbound*
\begin{proof}
	We establish for $\alpha\in H^1(M;\Z)$ first. 
	Take a taut least area surface $S$ dual to $\alpha$. By the interplay between the area and the Euler characteristic of a least area $S$ in [\citenum{bdNorms}, (3.3)], we have 
	$$\area(S) \leq 2\pi \chi_-(S) =2\pi \|\alpha\|_{Th}$$ 
	which implies $$\|\alpha\|_{Th} \geq \frac{1}{2\pi}\area(S).$$
	Using the monotonicity formula [\citenum{mrMinSurfaceShort}, Corollary 7(2)] and \eqref{tube radius lower bound} on an essential minimal disk component of $S\cap \T$, each disk has an area of at least $$\ds 2\pi (\cosh{R}-1) = 2\pi \left(\frac{0.29}{\sqrt{7.256}}\frac{1}{\sqrt{\lambda}}-1\right) \geq 2\pi\left(\frac{0.107}{\sqrt{\lambda}}-1\right).$$  Since there are at least $|\kappa|$ disks, 
	$$\area(S) \geq 2\pi|\kappa|\left(\frac{0.107}{\sqrt{\lambda}}-1\right)$$
	and \eqref{eq lower bound on Thurston norm sum} follows. By direct computation, if $\lambda \leq 58/10^6$, then $\frac{0.29}{\sqrt{7.256}}\frac{1}{\sqrt{\lambda}}-1>\frac{0.1}{\sqrt{\lambda}}$. 
	
	The proof for $\alpha\in H^1(M;\Q)$ follows from linearity, since both sides of \eqref{eq lower bound on Thurston norm sum} are linear with respect to $\alpha$. Then it extends to $H^1(M;\R)$ by the continuity of the Thurston norm and the Kronecker pairing. 
\end{proof}
In [\citenum{hwCSC}, Proposition 3.3], Huang and Wang have considered a similar inequality between the area of essential disks and genus. 
The above lemma also holds if we replace $\kappa(\alpha)$ the number of algebraic intersections with the number of geometric intersections which also consists of essential disks between  $S$ and $\T$. \par
One application is to apply \Cref{thurston norm lower bound by kappa / square root of length} to the homology $[S] \in H_2(M;\Z)$ of $M$ obtained from Dehn fillings. We can use, for example, [\citenum{nzVolHyp3fold}, Proposition 4.3] for the length of the short core geodesic $\gamma$. We demonstrate this philosophy by using \Cref{thurston norm lower bound by kappa / square root of length} to simplify the proof of [\citenum{bdNorms}, Theorem 1.4(b)], which also shows it is qualitatively sharp.
\begin{exa}\label{example LinkL14Thurston norm}
	There exists a sequence of closed hyperbolic $3$-manifolds $M_n$ whose injectivity radius $\lambda_n \rightarrow0$ and cohomology $\alpha_n$ such that 
		$$\|\alpha_n\|_{Th}\sim \bo\left(\frac{1}{\sqrt{\lambda_n}}\right).$$
\end{exa}
\begin{proof}
 Let $L=L14n21792$ be a two-component fibered link and $W$ be $S^3-L$, whose boundary consists of two tori $\partial W = T_1 \sqcup T_2$. We recall from [\citenum{bdNorms}, p.550] the following properties that $W$ enjoys. 
\begin{enumerate}
	\item $W$ is a cusped hyperbolic $3$-manifold. 
	\item The maps $\iota_k: H_1(T_k;\Z) \rightarrow H_1(W;\Z)$ are isomorphisms, because the two components of $L$ have linking number $1$. 
	\item $W$ has an orientation-reversing involution $\sigma$ that interchanges $T_k$ and acts on $H_1(W;\Z)$ by identity. In fact, $W$ is the orientation cover of the nonorientable census manifold $X=x064$. 
	\item $W$ fibers over the circle. 
\end{enumerate}
 Let $\alpha, \beta$ be two (primitive) fibered classes in a fibered cone $C_F$ which form an integral basis for $H^1(W; \Z)$, and let $a, b \in H_1(W; \Z)$ be the dual basis such that $\alpha(a)=\beta(b)=1$ and $\alpha(b)=\beta(a)=0$. Let $M_n$ be the closed manifold obtained by Dehn filling $W$ along $a−nb$ in $T_i$ and $\alpha_n \in H^1(M_n;\Z)$ the extension of $\widetilde{\alpha_n}=n\alpha+\beta$ to $M_n$. 
\begin{lem}\label{application simplify Brock-Dunfield Theorem 1.4 (b)}
	The Thurston norm of $\alpha_n$ grows at least linearly with respect to $n$. 
\end{lem}
\begin{proof}
	Since $n$ and $1$ are coprime, $n\alpha+\beta$ represents a single curve on $T_i$. Thus the surface $S_n$ dual to $\alpha_n$ has algebraic intersection number $1$ with the core geodesic $\gamma_i$. Using [\citenum{nzVolHyp3fold}, Proposition 4.3], the length of $\gamma_n$ is $\sim \frac{c}{n^2}$ for a uniform constant $c$ as $n \rightarrow \infty$. Using \Cref{thurston norm lower bound by kappa / square root of length}, we have, for $n$ large enough, \[\|\alpha_n\|_{Th} > 0.1\sqrt{c}\cdot n .\]
\end{proof}
To bound the Thurston norm of $\alpha_n$ from above, we show there exist surfaces in the homology class dual to $\alpha_n$ whose Thurston norms grow at most linearly. Note that $\|\cdot\|_{Th}$ is linear on the cone $C_F$. Thus $\|\widetilde{\alpha_n}\|_{Th}= n \|\alpha\|_{Th}+\|\beta\|_{Th}$. Since $S_n$ dual to $\alpha_n$ is obtained by capping off $\widetilde{S_n}$ dual to $\widetilde{\alpha_n}$ with two discs, we have $ \|\alpha_n\|_{Th}\leq |\chi(S_n)|=|\chi(\widetilde{S_n})|-2 =\|\widetilde{\alpha_n}\|_{Th}-2$, which implies that $\|\alpha_n\|_{Th}\leq n \|\alpha\|_{Th}+\|\beta\|_{Th}-2$. 
\end{proof}
\noindent \Cref{thurston norm lower bound by kappa / square root of length} does not give the precise asymptotic as in [\citenum{bdNorms}, p.550 i.], but the lower bound suffices for the proof of the upper bound on $\|\cdot\|_{L^2}/\|\cdot\|_{Th}$ in [\citenum{bdNorms}, Theorem 1.4]. Moreover, the above construction generalizes to $\alpha, \beta$ in a non-fibered cone and still produces cohomology $\alpha_n \in H^1(M_n;\Z)$ which grows linearly. 

\Cref{thurston norm lower bound by kappa / square root of length} also generalizes naturally to the cusp setting. Let $\alpha \in H^1(N; \Z)$ be dual to a taut $S\in H_2(N, \partial N;\Z)$. Suppose $N$ has $n$ cusp neighborhoods which are maximal (pairwise disjoint). Take a cusp neighborhood $E=T\times [1, \infty)$, whose metric with respect to the largest horotorus $T$ is given by $$ds^2_E=dt^2+e^{-2t}ds^2_T$$ for $t\geq 1$. Let $T_t$ be the horotorus at height $t$. We first define $\ell_E(\alpha)$ which is the length of the \textit{cohomology slope}, the shortest loops on $T$ in the kernel of $\alpha|_T$ (counting multiplicities). Then using the co-area formula and the foliation $S\cap T_t$ in $E$,  we will argue $\ell_E(\alpha)$ gives a lower bound on $\area(S\cap E)$. Suppose $\alpha|_T \in H^1(T;\Z)$ is equal to $k(\alpha_xdx+\alpha_ydy)$ where $\alpha_x$ and $\alpha_y$ are coprime. There exists a shortest loop $\gamma_{\alpha}\in H_1(T;\Z)$ in the kernel of $\alpha_xdx+\alpha_ydy$. With the intrinsic Euclidean metric $d_T$ on $T$, $\ell_E(\alpha)$ is defined as $$\ell_E(\alpha)\coloneqq k\ell_{d_T}(\gamma_{\alpha}).$$
  The intersection of $S$ and $E$ is foliated by loops on $T_t$, each of whose minimal length is given by $e^{-t}\ell_E(\alpha)$. Using the co-area formula, we have
   $$\area(S)\geq \int_1^{\infty}\ell_E(\alpha)e^{-t}dt = \frac{\ell_E(\alpha)}{e}.$$ Using the Gauss-Bonnet theorem, $ 2\pi \| \alpha\|_{Th} \geq \area(S)$. Using the thick-thin decomposition and considering the intersection of a taut surface $S$ with the short geodesics and maximal cusp neighborhoods, we obtain
\begin{lem}\label{thurstonNormLowerboundCusped}
	Let $N$ be a cusped hyperbolic $3$-manifold with $n$ maximal cusp neighborhoods $E_i$ and $m$ short geodesic $\gamma_i$ of length $\lambda_i$. Then the Thurston norm of $\alpha \in H^1(N;\Z)$ has a lower bound
	$$\|\alpha\|_{Th} \geq \frac{1}{2\pi e}\sum_{i=1}^{m} \ell_{E_i}(\alpha)+ \sum\limits_{i=1}^n \left|\int_{\gamma_i} \alpha\right| \left(\frac{0.107}{\sqrt{\lambda_i}}-1 \right).$$ 
\end{lem}	

\section{The proof of \Cref{theorem main conjecture BD}}\label{section the proof}
We start with some basic lemmas. Let $\T$ be a Margulis tube with a short geodesic $\gamma$ whose length is $\ell(\gamma)$. Let $\alpha$ be a harmonic $1$-form, with $\kappa(\alpha)=\int_{\gamma}\alpha$.
\begin{lem}\label{harmonic 1-form decomposition in tube}
On a Margulis tube $\T$, a harmonic $1$-form $\alpha$ admits an $L^2$-orthogonal decomposition
$$
\alpha=\frac{\kappa(\alpha)}{\ell(\gamma)}dz+df,
$$
where $f$ is harmonic and $\langle dz, df \rangle_{L^2(\T)}=0$. In particular, 
\begin{equation}\label{eq df L2 leq alpha L2}
	\|df\|_{L^2(\T)}\leq \|\alpha\|_{L^2(\T)}.
\end{equation}
\end{lem}
\begin{proof}
	Since $H^1(\T; \R)$ is $1$-dimensional and $\int_{\gamma}\frac{1}{\ell(\gamma)}dz=1$, the $1$-form $\alpha$ is cohomologous to $\frac{\kappa(\alpha)}{\ell(\gamma)}dz$. Thus 
	\begin{equation}\label{eq alpha-dz=df}
		\alpha -\frac{\kappa(\alpha)}{\ell(\gamma)}dz =df \quad \tn{ for some } f \in C^{\infty}(\T).
	\end{equation}
Since both $\alpha$ and $dz$ are harmonic (see [\citenum{bdNorms}, 6.1]), \eqref{eq alpha-dz=df} implies that $df$ is also harmonic and thus $d*df=0$. Denote $d^*$ the codifferential operator. Then $\Delta f= dd^*f+d^*df=-*d*df=0$ which implies that $f$ is harmonic. For $f(r,\theta, z)$, using \eqref{tubeONB1Forms}, we have
\begin{align*}
	df &= f_rdr+f_{\theta} d\theta +f_zdz, \\ 
	*df &= f_r(\sinh{r}\cosh{r}\di \theta \wedge dz)+f_{\theta}\frac{\cosh{r}}{\sinh{r}} dz\wedge dr + f_z \frac{\sinh{r}}{\cosh{r}}dr\wedge d\theta.  \\
\end{align*}
Moreover, 
\[
dz\wedge *df + *dz \wedge df = 2 f_z \frac{\sinh{r}}{\cosh{r}} dr \wedge d\theta \wedge dz. 
\]
Its integral on $\T$ is 
\begin{align*}
	\ds \int_{\T} dz\wedge *df + *dz \wedge df 
	&= 2\int_{0}^{R}\int_{0}^{2\pi} \int_{0}^{\ell(\gamma)} f_z\frac{\sinh{r}}{\cosh{r}}dr \wedge d\theta \wedge dz \\
	&= 2\int_{0}^{R}\int_{0}^{2\pi} f(r,\theta, \ell(\gamma))- f(r,\theta, 0) d\theta \wedge \frac{\sinh{r}}{\cosh{r}}dr\\
	&= 2\int_{0}^{R}\int_{0}^{2\pi} f(r,\theta+\theta_0, 0)- f(r,\theta , 0) d\theta \wedge \frac{\sinh{r}}{\cosh{r}}dr\\
	&=0,
\end{align*}
where $\theta_0$ is the twist parameter of the tube and we have used the harmonic expansion of $f$ in \eqref{harmonic expansion into hypergeometric}.  
Finally, let $c \coloneqq \frac{\kappa(\alpha)}{\ell(\gamma)}$. Since $\alpha =cdz+df$, we have  
\begin{align*}
	\ds \|\alpha \|_{L^2}^2 -\|df \|_{L^2}^2 &=  \int_{\T} (df +c dz )\wedge (*df +c*dz) -\int_{\T} df \wedge *df \\
	&= \int_{\T} cdz \wedge *df + c*dz \wedge df + c^2 dz \wedge *dz \\ 
	&=  \int_{\T} c^2 dz \wedge *dz \\
	&=\|cdz \|_{L^2}^2.
\end{align*}
\end{proof}

We now prove the main theorem \ref{theorem main conjecture BD}. 
\theoremMainConjectureBD*
	\begin{proof}
		We prove it for $\alpha \in H^1(M;\Z)$, and in general, it follows from extension by linearity to rational coefficients and then by continuity to real coefficients. Take a taut and least area surface $S$ dual to $\alpha$. \par
		We first assume $M$ only has one Margulis tube $\T$ and (for technical reasons) its maximal radius is $R+1$. Let $\T_R$ denote the subtube of radius $R$. 
	\begin{lem}\label{c}
	  There exists a uniform constant $c_1(V)$ such that
		$$
			\int_{S \cap \T_R} *df  \leq c_1(V) \area(S \cap \T_R) \|\alpha\|_{L^2(M)}.
	$$
	\end{lem}
\begin{proof}
		The function $|df|^2+f^2$ is subharmonic by [\citenum{hansHarmonicForms}, (4.10)] and hence satisfies the maximum principle for subharmonic functions. Thus we have 
		\begin{equation}\label{df L-infinity < boundary |df|^2+f^2 L-infinity}
			\|df\|^2_{\li(\T)} \leq \max\limits_{\T}(|df|^2+f^2) = \max\limits_{\partial \T}(|df|^2+f^2 ).
		\end{equation}
		Since \eqref{df L-infinity < boundary |df|^2+f^2 L-infinity} holds for any primitive $f$, by adding a constant if necessary, we can assume $\max\limits_{\partial \T} f=-\min\limits_{\partial \T} f$. Then we have \[2\max\limits_{\partial \T}f=\max\limits_{\partial \T}f-\min\limits_{\partial \T} f \leq \diam (\partial \T) \|df\|_{\li(\partial\T)}\]
		where $\diam (\partial \T)$ is the intrinsic diameter of the boundary torus, which is uniformly bounded above by a constant $d(V)$ by \Cref{uniform intrinsic injectivity radius and diameter}. Thus \eqref{df L-infinity < boundary |df|^2+f^2 L-infinity} implies 
		$$
	\|df\|^2_{\li(\T)} \leq  \left(\frac{d^2(V)}{4}+1\right)\|df\|_{\li(\partial \T)}^2.
	$$
Let $\delta$ be the injectivity radius of $\partial \T_R$ which is uniformly bounded from below by \Cref{R tube injectivity radius}. By making $\delta$  slightly smaller, at $p \in \partial \T_R$, the $\delta$-ball $B_{\delta}(p)$ is embedded and contractible, and also contained in $\T_{R+1}$. By \eqref{eq l infinity bounded by 5 / square root inj times L2}, there exists a constant $c=c(\delta)$ such that. 
		$$
		\|df\|_{\li(\partial\T_R)} \leq c(\delta) \|df\|_{L^2(B_{\delta})} \leq c(\delta) \|df\|_{L^2(\T_{R+1})}.
	$$
		Thus we have 
		$$
			\|df\|_{\li(\T_R)} \leq c(\delta) \sqrt{\frac{d^2(V)}{4}+1} \|df\|_{L^2(\T_{R+1})}
		$$
Renaming the constant $c(\delta) \sqrt{\frac{d^2(V)}{4}+1}$ by $c_1(V)$, we have
\begin{align*}
	\int_{S \cap \T_R} *df &\leq \area(S \cap \T_R) 	\|df\|_{\li(\T_R)} \\
\numberthis \label{integral *df leq area L2}	&\leq c_1(V) \area(S \cap \T_R) \|df\|_{L^2(\T_{R+1})} \\
	&\leq c_1(V) \area(S \cap \T_R) \|\alpha\|_{L^2(M)}
\end{align*}
where in the last step we have used \Cref{harmonic 1-form decomposition in tube}. 
\end{proof}
	 For $dz$, we adopt a different strategy. 
		\begin{lem}
			On a boundary-parallel disk $B$,
			\[
			\int_B *dz=0 \text{ and } \int_B *\alpha = \int_B *df. 
			\]
		\end{lem}
		\begin{proof}
			In cylindrical coordinates, the area two-form of $\partial \T_R$ is proportional to $dz \wedge d\theta$. It follows $ *dz = \frac{\sinh r}{\cosh r} dr \wedge d\theta=0$ \textbf{pointwisely} on $\partial \T$. The disk $B$ is homotopic onto a disk $B' \subset \partial\T$ relative to $\partial B$. Since $*dz$ is closed, by Stokes' theorem, we have $\int_B *dz=\int_{B'} *dz $ which is equal to $0$. The second equality follows from \Cref{harmonic 1-form decomposition in tube}. 
		\end{proof}
	\begin{lem}\label{dzEqualL2Normdz}
		For the intersection $S\cap \T$, we have
		\begin{equation}
			\ds \int_{S \cap \T} *\frac{\kappa(\alpha)}{\ell(\gamma)}dz =\left \| \frac{\kappa(\alpha)}{\ell(\gamma)}dz \right\|_{L^2(\T)}^2.
		\end{equation}
		\begin{proof}
		The intersection $S \cap \T$ consists of $|\kappa(\alpha)|$ least area essential disks $D$ and some boundary-parallel disks $B$. We already show $\int_B *dz=0$.
		If all essential disks are totally geodesic, then their area $2$-forms are $\sinh r \di r \wedge d\theta$. By  [\citenum{bdNorms}, (6.3)] and a direct computation, 
		\begin{align*}
			\ds \int_{S \cap \T} *\frac{\kappa(\alpha)}{\ell(\gamma)}dz 
			&=\kappa^2(\alpha)\frac{1}{\ell(\gamma)} \int_{0}^{R} \int_{0}^{2\pi} \frac{1}{\cosh r}(dr\wedge \sinh r \di \theta )\\
			&=\kappa^2(\alpha)\frac{2\pi}{\ell(\gamma)} \log \cosh{R} \\
			&= \frac{\kappa^2(\alpha)}{\ell^2(\gamma)}2\pi\ell(\gamma)\log \cosh{R}\\
			&= \| \frac{\kappa(\alpha)}{\ell(\gamma)}dz \|_{L^2(\T)}^2.
		\end{align*} 
		For the least area case, using a homotopy that slides along the boundary torus $\partial \T_R$, the least area disks are homotopic to totally geodesic ones. Since $*dz$ is closed and pointwisely zero along $\partial \T_R$, using Stokes' theorem, \Cref{dzEqualL2Normdz} holds for the least area $S$.
		\end{proof}
	\end{lem}
 Recall $R+1$ is the maximal radius of the Margulis tube $\T$ and suppose a slightly smaller Margulis constant $\epsilon$ gives rise to a Margulis tube $\T_R$ of radius $R$ (\Cref{R tube injectivity radius} implies the Margulis constant $\epsilon$ for the $R$-tube is $\geq  0.0045$). Let $\area(\cdot)$ denote the area functional and $(\cdot)_{\geq \epsilon}$ denote the $\epsilon$-thick part. Using a thick-thin decomposition with Margulis constant $\epsilon$, \eqref{eq l infinity bounded by 5 / square root inj times L2}, \Cref{harmonic 1-form decomposition in tube}, and \Cref{dzEqualL2Normdz}, we have	

		\begin{align*}
			\|\alpha\|^2_{L^2} & =\int_{S} *\alpha =\int_{S \cap (M)_{\geq \epsilon}} *\alpha + \int_{S \cap (M)_{\leq \epsilon}} *\alpha\\ 
			&=\int_{S \cap (M)_{\geq \epsilon}}*\alpha+\int_{S \cap \T_R} *df +\int_{S \cap \T_R} *\frac{\kappa(\alpha)}{\ell(\gamma)}dz  \numberthis \label{mainDerivation}\\
			&\leq c_2(\epsilon) \|\alpha \|_{L^2} \area((S)_{\geq \epsilon}) + c_1(V) \|\alpha\|_{L^2} \area((S)_{\leq \epsilon})+\frac{2\pi\kappa^2(\alpha)}{\ell(\gamma)} \log \cosh{R} \\ 
			& = c_3(V) \|\alpha \|_{L^2} \area(S) +\kappa^2(\alpha)\frac{2\pi}{\ell(\gamma)} \log \cosh{R} \\
			&\leq c_3(V) 2\pi \|\alpha \|_{L^2} \|\alpha \|_{Th} + 200\pi \|\alpha \|_{Th}^2 \log \cosh{R},  \numberthis \label{mainUpperbound}\\
{\tiny }		\end{align*} 
		where in the last inequality, we have used $\area(S) \leq 2\pi \|\alpha\|_{Th}$ from [\citenum{bdNorms}, (3.2)] and \Cref{thurston norm lower bound by kappa / square root of length} which bounds the Thurston norm from below by the pairing $\kappa$. Solving the quadratic inequality, we then have
		\begin{equation}\label{eq l2 thurston cosh R}
				\ds \frac{\| \alpha \|_{L^2}}{\| \alpha \|_{Th}} \leq  c_3(V)\pi + \sqrt{c_3(V)^2\pi^2 + 200\pi \log \cosh{R}}.
		\end{equation}		
	By the volume formula for the tube,
	$$ \pi \lambda (\cosh^2 R-1) <V$$
	which implies 
	$$\log \cosh R < \frac{1}{2} \log \left(\frac{V}{\pi\lambda}+1\right).$$ 	
	Thus \eqref{eq l2 thurston cosh R} rewritten in terms of the injectivity radius $\lambda$ is 
	\begin{align*}
			\ds \frac{\| \alpha \|_{L^2}}{\| \alpha \|_{Th}} &\leq  c_3(V)\pi + \sqrt{c_3(V)^2\pi^2 + 100\pi  \log \left(\frac{V}{\pi\lambda}+1\right)}	\\
			&\leq 2	\sqrt{c_3(V)^2\pi^2 + 100\pi  \log \left(\frac{V}{\pi\lambda}+1\right)}	
	\end{align*}
	If we rename $c_3(V)$ as $c(V)$, this proves \Cref{theorem main conjecture BD} when $M$ only has one Margulis tube. If $M$ has $n$ tubes $\T_{i}$ with the length of the core geodesic $\ell(\gamma_{i})\leq 58/10^6$, then the right-hand side of \eqref{mainDerivation} is replaced with the following: 
		\[
		\int_{S \cap (M)_{\geq \epsilon}} *\alpha + \sum_{i=1}^{n}\int_{S \cap T_{i}} *df_{i} + \sum_{i=1}^{n}\int_{S \cap T_{i}} *\frac{\kappa_i(\alpha)}{\ell(\gamma_{i})}dz_{i}.
		\]
		Note that in dimension $3$, the thick part $(M)_{\geq \epsilon}$ of a hyperbolic manifold is connected. Since different Margulis tubes are disjoint and $\area((S)_{\leq \epsilon}) = \sum_{i=1}^n  \area(S \cap \T_{i})$, the second term is again bounded from above by $c_3(V) \|\alpha\|_{L^2} \area((S)_{\leq \epsilon})$. Suppose $\T_{1}$ denotes the Margulis tube with the shortest geodesic, with the largest radius $R_{1}+1$. We have
		\begin{align*}
			\sum_{i=1}^{n}\int_{S \cap T_{i}(R_{i})} *\frac{\kappa_i(\alpha)}{\ell(\gamma_{i})}dz_{i}&=2\pi\sum_{i=1}^{n}\frac{\kappa_i^2(\alpha)}{\ell(\gamma_{i})} \log \cosh{R_{i}} \\
		(\tn{since } R_{i}\leq R_{1}) \quad	&\leq 2\pi\log \cosh{R_{1}} \sum_{i=1}^{n}\frac{\kappa_i^2(\alpha)}{\ell(\gamma_{i})} \\
			&\leq 2\pi\log \cosh{R_{1}} \left(\sum_{i=1}^{n}\frac{\kappa_i(\alpha)}{\sqrt{\ell(\gamma_{i})}}\right)^2 \\ 
		(\tn{by } \eqref{eq asymptotic lower bound on Thurston norm}) \quad	&\leq 200\pi \|\alpha\|_{Th}^2\log \cosh{R_{1}} .
		\end{align*}
		Thus the sum of the three terms is bounded by the right-hand side of \eqref{mainUpperbound} just like the one-tube case. Finally, note that the injectivity radius of $M$ equals $\ell(\gamma_{1})/2$.
	\end{proof}
\begin{rem}
	When the dimension is two, a Margulis tube $\T$ can be separating. If $\alpha$ is supported in a component of the complement of $\T$, the norm of $\alpha$ across the tube satisfies a certain decay rate \cite{bmmsEnergyharmonic1-formSurface}. When the dimension is three, although the Margulis tubes are non-separating, their boundary still serves a similar role. The essential intersection in the thin part contributes to the conjectural growth rate while the estimates in the thick part stay uniformly bounded relative to the area in view of \eqref{eq l infinity bounded by 5 / square root inj times L2}.
\end{rem}
Without the essential intersections in the thin part, \Cref{theorem main conjecture BD} can be improved significantly. 
	\begin{cor}\label{linearThurL2Inequality}
		If all classes in $H^1(M;\Z)$ pair trivially with the geodesics of length  $\leq 58/10^6$,  \Cref{theorem main conjecture BD} simplifies with the same constant $c(V)$:
		$$
			\|\cdot\|_{L^2} \leq 2\pi c(V)  \|\cdot \|_{Th} \tn{ on } H^1(M; \R).
		$$
	\end{cor}
	\begin{proof}
		The trivial pairing implies \eqref{eq alpha-dz=df} simplifies and becomes $\alpha=df$ in the Margulis tubes. As a consequence, \eqref{mainDerivation} simplifies (the term with $*dz$ vanishes) and gives 
		$$
			\|\alpha\|_{L^2} \leq 2\pi c(V)  \|\alpha \|_{Th}.
		$$
		\end{proof}	
\noindent	This motivates \Cref{def homologically thick}. 
\section{The generic growth and \ehl}\label{section generic growth}
Recall that a closed incompressible surface $S$ in a cusped  $N$ is called peripheral if $S$ is isotopic to a boundary component of $N$. One goal of this section is to show that a least area closed incompressible surface $S$ in $M$ obtained from a ``generic" Dehn filling of $N$ ``looks like" a nonperipheral surface (\Cref{generic Dehn fillings have surfaces supported in the thick part}). We start by setting up some basic definitions. Let $\epsilon$ be a Margulis constant of $M$. 
\begin{defi}\label{def homologically thick}
 A closed hyperbolic $3$-manifold $M$ is \textbf{$\epsilon$ homologically thick} if and only if every closed geodesic $\gamma$ of length $\leq \epsilon$ is $0$ in $H_1(M; \Q)$. 
\end{defi}
\noindent This is equivalent to saying that any $[S] \in H_2(M)$ has trivial algebraic intersections with the tubes $(M)_{\leq \epsilon}$ by the \pc duality.
\noindent The notion of \ehl is ``generic". 
\begin{defi}[A generic Dehn filling]
	Let $N$ be a finite-volume hyperbolic manifold with $n$ cusps and let $(p_1/q_1, \cdots, p_n/q_n)\in T^n$ be the space of Dehn filling coefficients. By a \textbf{generic} Dehn filling, we mean Dehn fillings $(p, q) \in T^n-X(N)$ where $X(N)$ is a piecewise smooth finite subcomplex of codimension at least $1$.
\end{defi}
\begin{defi}[A generic manifold of volume $\leq V$]\label{def generic manifold of volume less than V}
	By the \tj theory, all closed hyperbolic $3$-manifolds come from Dehn fillings of finitely many cusped manifolds $N_1, \cdots, N_n$. For each $N_i$, we remove one codimension-one complex from the Dehn filling space, and we call the rest of the hyperbolic Dehn fillings generic.
\end{defi}

	Thurston's hyperbolic Dehn surgery theorem \cite{thurstonNotesgeometrytopology3manifolds} asserts that except for finitely many slopes in $T^n$, all Dehn fillings are closed hyperbolic. Hatcher's theorem \cite{haBoundaryCurveIncom} shows that the subset $BCS(\partial N)$ of the curve system $CS(\partial N)$ that bounds incompressible and boundary-incompressible surfaces in $N$ is contained in the image of $H$ which is the union of finitely many rank-$n$ subgroups of $H_1(\partial N) \cong \Z^{2n}$, since there are finitely many branched surfaces which carry all incompressible and boundary-incompressible surfaces $S$ in $N$ \cite{foIncompressibleBranchedSurfaces} and $\partial S$ belongs to a self-annihilating subspace of $\Z^{2n}$. The image of $H$ is then used to define the finite subcomplex $X(N)$ of $T^n$. If $M$ is obtained from a Dehn filling of $N$, an incompressible surface $S$ in $M$ that has a nonzero geometric intersection with $(M)_{<\epsilon}$ comes from a properly embedded surface $S'$ with boundary in $N$. Together we have for a generic Dehn filling $M$, $H_2(M)$ are represented by surfaces that have a trivial geometric intersection with $(M)_{<\epsilon}$ and $M$ is $\epsilon$ homologically thick. Thus we have proven
\begin{lem}\label{generic Dehn filling homologically thick}
	A closed hyperbolic $3$-manifold $M$ obtained from a generic Dehn filling of $N$ is $\epsilon$ homologically thick. In particular, any closed incompressible surface $S \subset M$ is homotopic to surfaces $\subset (M)_{\geq\epsilon}$.
\end{lem} 
\begin{cor}
	Let $N$ be an orientable hyperbolic $3$-manifold with one cusp. All but finitely many Dehn fillings are generic. 
\end{cor}
A more complicated example with two cusps is the following. 
\begin{exa}
	All non exceptional Dehn fillings with slopes other than $\{(p/q, q/p)\}$ on \lfourteen are $\epsilon$ homologically thick. In fact, they are all homology spheres.   
\end{exa}
\begin{proof}
	We continue using the notations from \Cref{example LinkL14Thurston norm}. All homology here are integral coefficients. Let $\sigma: W \rightarrow W$ be the involution which interchanges the two boundary components $T_k$ and acts on $H_1(W)$ as identity. By the Lefschetz duality, $\sigma$ acts on $H_2(W, \partial W)$ as identity as well. Thus every element $S \in H_2(W, \partial W)$ must have a boundary component on both $T_k$. Since the two components of the link have linking number $1$, the meridian and longitude of $T_1$ are the longitude and meridian of $T_2$, respectively. The involution $\sigma$ maps $(p,q)$ corresponding to $\partial S \cap T_1$ to $(q, p)$ corresponding to $\partial S \cap T_2$. Thus all slopes which bound surfaces representing $H_2(W, \partial W)$ are of the form $(p/q, q/p)$. 
\end{proof}
Another natural class of \ehl manifolds is the following.
\begin{exa}
	Let $S$ be a closed orientable surface of genus $g\geq 2$, and $\gamma$ be a separating loop on $S$. Let $\phi \in Homeo(S)$ be a pseudo-Anosov such that $\phi_*:H_1(S) \rightarrow H_1(S)$ has irreducible characteristic polynomial over $\mathbb{Z}$ with a real root $>1$. Denote $D_{\gamma}$ a Dehn twist with respect to $\gamma$. Then $\phi \circ D^j_{\gamma}$ is a pseudo-Anosov and the injectivity of the mapping torus $M_j \coloneqq M_{\phi \circ D^j_{\gamma}} \rightarrow 0$ as $j \rightarrow \infty$, all with $b_1=1$ and $H_2$ generated by $S$. In particular, for all $j$ large enough, $M_j$ is $\epsilon$ homologically thick. 
\end{exa}
\begin{proof}
	\cite{ND3mfFibering, thurstonNormHomology3manifold} imply that if we choose a diffeomorphism which fixes no elements in $H_1(\Sigma; \Z)$, then the mapping torus has betti number $1$. The effect of composing Dehn twists $D^j_{\gamma}$ with $\phi$ corresponds to removing a tubular neighborhood of $\gamma$ and then fill with slopes $\gamma+j\cdot \text{meridian}$ ([\citenum{rdKnotsAndLinks}, 9.H]) which by Thurston's hyperbolic Dehn surgery theorem results in closed hyperbolic manifolds for all large $j$. 
\end{proof}

\begin{theorem}\label{generic L2 thurston norm linear}
		Let $V>0$. Then for a generic closed hyperbolic $3$-manifold $M$ whose volume $\leq V$, there exists a constant $c(V)$ such that 
	\begin{equation}
	\frac{\pi}{\sqrt{V}} \|\cdot\|_{Th}   \leq	\|\cdot\|_{L^2} \leq  c(V)\|\cdot\|_{Th}	\tn{ on } H^1(M;\R). 
	\end{equation}
\end{theorem}
\begin{proof}
The inequality is vacuous if $b_1(M)=0$. Appealing to \Cref{def generic manifold of volume less than V}, generic manifolds satisfy the condition of \Cref{linearThurL2Inequality}. Thus the conclusion follows. 

The left-hand side follows naturally from the left-hand side of \eqref{bdi} and $\frac{\pi}{\sqrt{V}}\leq \frac{\pi}{\sqrt{\vol(M)}}$. 
\end{proof}
\noindent This shows that most sequences of $M_j$ with $\vol(M_j)\leq V$ and $\inj(M_j)\rightarrow 0$, the $L^2$-norm and the Thurston norm are uniformly comparable instead of behaving like $\|\cdot\|_{L^2}/\|\cdot\|_{Th} \sim O(\sqrt{-\log \inj(M_j)})$ as in Brock-Dunfield examples \eqref{eq conjecture}.  
For a generic Dehn filling $M$, natural representatives of closed incompressible surfaces supported in the thick part $(M)_{\geq\epsilon'}$ are their least area representatives. 
\begin{cor}\label{generic Dehn fillings have surfaces supported in the thick part}
	Let $M_j$ be a generic sequence of closed hyperbolic $3$-manifolds such that $\vol(M_j) \leq V$. Let $S_j$ be a sequence of least area taut surfaces whose area tend to infinity. Then there exists $\epsilon' >0$ such that for all large $j$,  the surfaces $S_j \subset (M_j)_{>\epsilon'}$. 
\end{cor}
\begin{proof}
Recall that finitely many manifolds $N_1, \cdots, N_n$ generate all manifolds of volume $\leq V$ via Dehn filling. First suppose $M_j \rightarrow N_1$. Then any incompressible surface $S\subset M$ has a trivial geometric intersection with the filled tubes by the definition of generic manifold. By \Cref{classification surfaces in tube}, $S$ intersects a Margulis tube $\T$ of $M$ in terms of incompressible annuli and boundary-parallel disks. In the three paragraphs before [\citenum{fpMinimal-area-surfaces-fibered-3-manifolds}, Remark 2.4], Farre and Vargas Pallete prove that with a uniform lower bound on $twist^2 / length = \theta_0^2 / \lambda$ of $\T$, a least area surface $S$ cannot intersect $\T$ into an incompressible annulus. In the proof of [\citenum{fpMinimal-area-surfaces-fibered-3-manifolds}, Theorem 2.5], they show $M_j$ which converges geometrically satisfies this lower bound condition for large enough $j$. Thus $S_j \cap \T_j$ consists of boundary-parallel disks, which are uniformly close to $\partial T_j$ by \Cref{uniform-depth-boundary-parallel-disks}. Thus all surfaces $S_j \subset (M_j)_{\geq \epsilon_1}$. 

Let $\epsilon' =\min\{\epsilon_1, \cdots, \epsilon_n\}$, the conclusion follows. 
\end{proof}
\noindent \cite{fpMinimal-area-surfaces-fibered-3-manifolds} proves the same result assuming the areas of $S_j$ are uniformly bounded from above. 
We now provide another proof of the above Corollary, which views $S_j$ from infinity. For the simplicity of the notations we assume $N$ only has one cusp. We control the support of $S_j \subset M_j$ via the geometry and topology of $N$ and geometric convergence. By the definition of $M_j \rightarrow N$, we have
	\begin{enumerate}
		\item $\sigma_j >0$ such that $\sigma_j \rightarrow 0$,
		\item $k_j >1$ such that $k_j \rightarrow 1$,
		\item a map $\psi_j: N \rightarrow M_j$ which restricts to a $k_j$-biLipschitz diffeomorphism from a neighborhood of $N_{\geq \epsilon}$ into $M_j$, such that $\psi_j(N_{\geq \epsilon})$ is contained in the interior of $(M_j)_{\geq \epsilon}$ and $\psi_j(\partial M_{\geq \epsilon})$ is disjoint from a neighborhood of $\psi_j(M_{[\epsilon+\sigma_j, \infty)})$.		
	\end{enumerate}
 Ruberman in [\citenum{rdMutationVolumeKnot}, Section 3] establishes the existence of least area representatives of both closed nonperipheral and cusped surfaces. Huang and Wang in [\citenum{hwClosedMinCusped}, Section 2.2] define a constant $\tau_3$ so that the thick part $N_{\geq\tau_3}$ contains all least area representatives of nonperipheral surface [\citenum{hwClosedMinCusped}, Corollary 1.2] (see also [\citenum{chmrMinNoncompactExist}, Theorem B]). Replace $\epsilon$ by the smaller of $\epsilon$ and $\tau_3$, which we still denote by $\epsilon$. For any closed incompressible surface $S_j$, since a generic Dehn filling $\psi_j$ annihilates surfaces with boundary and surfaces homotopic to $\partial N$, there is an closed nonperipheral $S$ supported in $N_{\geq \epsilon}$, such that 
\begin{center}
	$S_j=\psi_{j*} S$.
\end{center}
	Let $B_j \subset S_j \cap \T_j$ be a disk and $A_j\subset S_j \cap \T_j$ be an annulus. If the radius of the largest geodesic disk in $B_j$ is $r$ with center $x$, then by the Gauss-Bonnet theorem (see [\citenum{cmCourseMinSurface}, Lemma 2.6]) the area of the disk is $>\pi r^2$ since the sectional curvature of a minimal surface in $M$ is $\leq -1$. The boundary $\partial B_j$ bounds a disk $D'_j \subset \partial \T_j$, whose area is bounded from above by the area of $\partial N$ up to a uniform constant. This directly bounds the radius of $B_j \subset S_j$ which also bounds the depth of $B_j$. This is essentially [\citenum{bwMinSurQuasiFuch}, Lemma 6], where he uses $\partial M_j$ as a barrier surface which comes from a classical exchange argument of Schoen-Yau. 
	
  Let $\phi_j \coloneqq \psi_j^{-1}$ defined in a small neighborhood of $(M_j)_{ \geq \epsilon}$ which restricts to a $J$-biLipschitz diffeomorphism $\phi_j: (M_j)_{\geq \epsilon} \rightarrow N_{\geq \epsilon /1.2}$, by [\citenum{fpsEffectiveBilipschitzBounds}, Theorem 1.2]. Thus $\phi_j'(A_j)$ is an annulus in the cusp $\mathcal{C}$ of $N$, supported between $T \times t_1$ and $T\times t_2$, where $T$ is a horotorus for $\mathcal{C}$, and $t_i= \log(iL_0)$, $L_0$ the constant in [\citenum{hwClosedMinCusped}, Section 2.1]. Now we have 
$$
	\area(A_j) \geq \frac{1}{k_j^2} \area(\phi_j'(A_j))	> \area(T \times t_1)
$$
as in the proof of [\citenum{hwClosedMinCusped}, Theorem 5.9]. As $j \rightarrow \infty$, 
$$
	\area(T \times t_1) \rightarrow \area(\partial M_j).
$$ 
We can then argue as in [\citenum{hwClosedMinCusped}, Corollary 5.7, Theorem 5.9] that there exists a $t_2$ in terms of the cusp shape of $N$ so that if the surface $S_j$ penetrates $\psi_j(T \times [t_1, t_2])$ then it is not least area. 
\begin{rem}
	The alternative proof provides the original motivation with which we think \Cref{generic Dehn fillings have surfaces supported in the thick part} should hold. It demonstrates that generically the least area closed incompressible surfaces in $M$ are $k_j$-biLipschitz to nonperipheral surfaces in $N$. 
\end{rem}
[\citenum{kahn2023geometrically}] studies the distribution of random surfaces whose genera tend to infinity in a closed hyperbolic manifold. If $M_j\rightarrow N$ with a sequence of surfaces $S_j \subset M_j$ such that $\tn{genus}(S_j) \rightarrow \infty$, a related question is the limit distribution of $S_j$. Generically, nonperipheral surfaces in $N$ control $S_j$ for large $j$. Thus we first need to study the distribution of nonperipheral surfaces whose genera tend to infinity.

  A random point of view of $3$-manifolds appears in the work of Dunfield and Thurston \cite{dtRandom3manifolds} in an attempt to test important open conjectures back then such as the virtual Haken and virtual betti number conjectures which were all resolved by Agol. See, for example, its connection with Heegaard splitting \cite{mjRandomHeegaardSplitting}, quantum invariants \cite{dwQuantumRandom3manifolds}, fibered commensurability \cite{mhFiberedCommensurabilityRandomMapping}, torsion growth \cite{bbghExponentialTorsionGrowth}, and volumes  \cite{vgVolumeRandom3-manifold} and the references therein for more related topics. In two dimensions, there is also wonderful ongoing work regarding topology, geometry and eigenvalues of random hyperbolic surfaces, see for example \cite{wxgafa, nwx, yyPrimeGeodesicTheorem, mmGrowthWeilPeterssonRandom} and the references therein for more related topics. \par 

\section{The geometry of best Lipschitz maps, geodesic laminations, and fibrations}\label{section lipschitz constant and entropy}
For this section, $M^n$ is a closed hyperbolic $n$-manifold and $f: M^n \rightarrow S^1=\R/\Z$ denotes a continuous circle-valued map. Let $[d\theta]\in H^1(S^1;\Z)$ be a generator and $\alpha\in H^1(M;\Z)$. We say $f \in \alpha$ if $[f^*d\theta]=\alpha$.
\subsection{The Lipschitz constants and best laminations}
 We first define the Lipschitz constant of a homotopy class. 
\begin{defi}\label{def Lipschitz constant of homotopy}
	For $f: M \rightarrow S^1$, the Lipschitz constant of an open subset $U$ is
	$$\lip_f(U) \coloneqq \inf\{L \in \R: d(f(x), f(y)) \leq Ld(x, y), \forall x, y\in U\},$$
	with the convention that $\inf \emptyset =\infty$. The Lipschitz constant  at $x$ is $$\lip_f(x)\coloneqq \lim\limits_{r\rightarrow 0}\lip_f(B_r(x)).$$ The Lipschitz constant  of $f$ is $$\lip_f\coloneqq \lip_f(M).$$
	 The Lipschitz constant of the homotopy class of $f$ is $$\lip_{[f]}= \min\limits_{h\sim f} \lip_h. $$
\end{defi}
\noindent If $f$ is $C^1$, then $\lip_f = \sup\limits_{x\in M} |df_x|$.
\begin{defi}[$\| [\alpha{]} \|_{\li}$]\label{def Lipschitz constant of cohomology}
	Let $\alpha \in H^1(M; \Z)$. The $L^{\infty}$-norm on the cohomology class $[\alpha]$ is  
	\[ \|[\alpha]\|_{\li} \coloneqq \min\limits_{\beta \sim \alpha \in H^1(M; \Z)} \max_{x\in M} |\beta_x|. \]
\end{defi}
\noindent  \dau identify $\|[\alpha]\|_{\li}$ with the Lipschitz constant $\lip_{[f]}$ in [\citenum{dubestLipLG}, Proposition 2.2]. If $f \in \alpha$ is infinity harmonic, [\citenum{dubestLipLG}, Theorem 2.4] shows that 
  \begin{equation}\label{map minimizing Lipschitz constant}
  	  \lip_f = \lip_{[f]}.
  \end{equation}
In general, a map $f$ minimizing the Lipschitz constant in its homotopy class (satisfying \eqref{map minimizing Lipschitz constant}) is called \textbf{best Lipschitz} \cite{dubestLipLG}. The existence of best Lipschitz maps follows from the compactness of $M$ and $S^1$ and the Arzel\`a-Ascoli theorem. \dau interpret the Lipschitz constant in terms of the length of geodesics and the Kronecker pairing. Let $\fancyS$ be the set of all free homotopy classes of simple closed curves in $M$ and $\gamma \in \fancyS$. Denote $\ell(\cdot)$ the length functional. A \textbf{geodesic lamination} is a closed subset of $M$ which is a disjoint union of simple and complete geodesics that is called its \textit{leaves}.
\begin{theorem}[\citenum{dubestLipLG}, Theorem 5.8, 5.2]\label{intersection over length sup = Lipschitz constant K'=K existence of msl}
	Define
	 \[K(\gamma)\coloneqq  \frac{|\int_{\gamma }\alpha|}{\ell(\gamma)}.\] 
	and
	\[K \coloneqq\sup\limits_{\gamma \in \mathcal{S}} K(\gamma).\] 
	 Then we have
	$$
		K= \lip_{[\alpha]}=\lip_{[f]}.
$$
  Moreover, there is a geodesic lamination $\Lambda(f)$ realizing the supreme: there exists $\gamma_n \in \fancyS$ such that 
  $$\gamma_n \rightarrow \Lambda(f) \tn{ and } K(\gamma_n) \rightarrow K.$$ The leaves $\gamma$ of $\Lambda(f)$ have a unit-speed parameterization such that $$\gamma'(t) = \frac{1}{K}\nabla f (\gamma(t)).$$ The restriction of $df$ to the tangent bundle of $\gamma$ is the uniform multiplication by the Lipschitz constant $K$: $|df|_{\Lambda(f)}=|\nabla f|_{\Lambda(f)}=K$. 
\end{theorem}
On a surface, a geodesic lamination can have uncountably many leaves. One way to interpret $ \int_{{\Lambda(f)}} \alpha/\ell(\Lambda(f))=K(\Lambda(f))$ is that for any geodesic segment $\gamma \subset \Lambda(f)$, 
\begin{equation}\label{eq interpretation intersection and length of a lamination}
	\frac{|\int_{\gamma} \alpha|}{ \ell(\gamma)}=K.
\end{equation}
A circle-valued map $f$ is called \textbf{tight} \cite{flmMinimizingLaminationsRegularCover} if its Lipschitz constant is equal to $K$. The definition of tight maps and best Lipschitz maps are equivalent in view of \Cref{intersection over length sup = Lipschitz constant K'=K existence of msl}. [\citenum{flmMinimizingLaminationsRegularCover}, Proposition 3.1] shows that a subset of the maximal stretch locus whose pullback to the frame bundle is invariant under the geodesic flow is a geodesic lamination, which we will call \textbf{the \msl of $f$}. 
 \begin{defi}[Maximal stretch locus and maximal stretch lamination]\label{maximal stretch locus and lamination}
 	Let $f: M\rightarrow S^1$ be a map. The maximal stretch locus of $f$ is $\{x\in M: \lip_f(x)= \lip_{f}\}$. When $M$ is a hyperbolic, for $f$ a tight map, a \textbf{\msl}$\Lambda=\Lambda(f)$ is the subset of the maximal stretch locus of $f$ which is a geodesic lamination. 
 \end{defi}
 \dau use comparison with cones for an infinity harmonic map to show its maximal stretch locus forms a geodesic lamination. \cite{fsExistenceCriticalSubsolutions} shows that there exists a $C^1$ tight map in the homotopy class of $f$. Rudd in [\citenum{rcStretchLamination}, Proposition 4.3] proves all maximal stretch
laminations are orientable. For other works on maximal stretch laminations, see \cite{gkMaximallyStretchedLaminations,daskalopoulos2022analytic,pan2022ray} and the references therein. For definitions of the topological entropy, see [\citenum{flp}, Chapter 10, 14].
A simple observation is the following.
\begin{lem}
	The cohomology $[\alpha]$ is trivial if and only if the Lipschitz constant $\lip_{[\alpha]}$ is $0$.
\end{lem}
\begin{proof}
	$\Rightarrow$. 
	If $\alpha=0$, then $f$ is homotopic to a constant map $f_c$ whose Lipschitz constant is $0$. \par 
	$\Leftarrow$ 
	If $\lip_{[\alpha]}=0$, then by the definition of $\lip_{[\alpha]}$, for any $\epsilon>0$, there exists an $f:M \rightarrow S^1$ such that the length of the image $f(M)$ is $\leq \epsilon \cdot \tn{diameter}(M)$. For $\epsilon$ small enough, the circle-valued map $f$ is homotopic to a constant map. 
\end{proof} 
	\begin{lem}\label{level set of tight cut lamination equally}
	For $n\geq 2$, let $f: M^n \rightarrow S^1$ be a tight map and $\Lambda(f)$ be its maximal stretch lamination. Every geodesic segment of $\Lambda(f)- f^{-1}(\theta)$ has a length of $1/K$. 	
	\end{lem}
	\begin{proof}
	Let $S = f^{-1}(\theta) \subset M$ be a (probably disconnected) level set (we do not have enough regularity to argue $S$ is a submanifold, but $S$ being compact suffices). Let $\gamma_0 \subset \Lambda(f)-S$ be a connected geodesic segment which inherits an orientation of $\nabla f$, from $p_0 \in  \gamma_0\cap S$ to $p_1 \in \gamma_0 \cap S$. The two endpoints $p_0, p_1$ can belong to two adjacent components of $S$. We have 
	$$\int_{{\gamma}_0} d{f} = {f}({p}_1) - {f}({p}_0)  =1$$ and 
	$$	\int_{{\gamma}_0} d{f} = \int_{0}^{\ell({\gamma}_0)} {\gamma}'(t)\cdot \nabla {f} \di t = \int_{0}^{\ell({\gamma}_0)} |\nabla {f}| \di t = \ell({\gamma}_0) \cdot K,$$
	where we have used the fact that any leaf of \msl is parallel to $\nabla f$ whose norm  $|\nabla {f}|=K$. 
	This implies 
	$$\ell({\gamma}_0) =\frac{1}{K}.$$ 
	\end{proof}
	The intersection of the maximal stretch laminations $\Lambda(f)$ of all tight maps $f$ for a cohomology $\alpha$ is nonempty by [\citenum{gkMaximallyStretchedLaminations}, Theorem 1.3], and denoted by
	$$\Lambda(\alpha)= \bigcap\limits_{f \tn{ tight}}\Lambda(f).$$
		Let $\widetilde{M}$ be the $\Z$-cover of $M$ such that the following diagram commutes, and $\tau$ is a deck transformation on $\wm$ such that $M=\wm/\langle \tau \rangle$ 	
	\[
	\begin{tikzcd}
		\wm \arrow{r}{\widetilde{f}} \arrow{d}{\pi}  & \R \arrow{d}{\pi} \\
		M \arrow{r}{f} & S^1.
	\end{tikzcd}
	\]
	Let $\gamma \subset \Lambda(f)$ be a leaf and $\widetilde{\gamma} \subset \wm$ be its cover. By \Cref{intersection over length sup = Lipschitz constant K'=K existence of msl}, the restriction of $\tilde{f}:\widetilde{\gamma} \rightarrow \R$ is a uniform multiplication by the Lipschitz constant $K$. Moreover, $\widetilde{\gamma}'$ is a constant multiple of $\nabla f$: $\widetilde{\gamma}'(t) =\frac{1}{K}\nabla \tilde{f}(\widetilde{\gamma}(t))$. 
	In the proof of [\citenum{flmMinimizingLaminationsRegularCover}, Lemma 3.2], \flm describe $\tilde{f}$ as a ``ruler for coarsely measuring distances in $\wm$". Another useful point of view is that for $\tilde{p}, \tilde{q} =\tau(\tilde{p}) \in \wm$, or $\tilde{p}, \tilde{q} \in \widetilde{\Lambda(\alpha)}$, $\tilde{f}(\tilde{q})- \tilde{f}(\tilde{p})$ measures the ``algebraic intersection" of a segment from $\tilde{p}$ to $ \tilde{q}$ to a submanifold dual to $\alpha$. 
	\begin{defi}
	Let $\tilde{\eta} \subset \wm$ be a segment from $\tilde{p}$ to $ \tilde{q}$ such that $\pi(\tilde{\eta})$ is a closed loop or $\tilde{\eta} \subset \Lambda(\alpha)$. We define the pairing between $\tilde{\eta}$ and $\alpha$ as 
	$$\tilde{f}(\tilde{q})- \tilde{f}(\tilde{p}).$$
	\end{defi}
\noindent	When $\pi(\tilde{\eta})$ is a closed loop, $\langle \pi(\tilde{\eta}), \alpha \rangle =\int_{\pi(\tilde{\eta})}df=\tilde{f}(\tilde{q})- \tilde{f}(\tilde{p})$ for any $f$ representing $\alpha$.  \par 
	Recall that the definition of the Lipschitz constant of a homotopy class comes from min-max:
	$$ \lip_{[f]}= \min\limits_{h\in [f]} \max\limits_{x\in M} \lip_h(x).$$ 
	The max implies that $1/K$ is the distance (the smallest length of a geodesic segment) between a level set $\widetilde{S}_0$ and its translate $\tau \cdot \widetilde{S}_0$ of a tight maps (this holds even if $\widetilde{S}_0$ has several components), while the min implies that within the homotopy class this ``fiber translation length" is maximal for a tight map. We now formalize this idea. 
 
	\begin{defi}\label{fiber translation length of monodromy phi}
Let $\alpha\in H^1(M;\Z)$ and $f\in \alpha$. We define the fiber translation length of a level set $\stheta=f^{-1}(\theta)$ as the distance between $\widetilde{\stheta}$ and $\tau \cdot \widetilde{\stheta}$ in the $\Z$-cover $\wm$:
		$$
		d(\stheta)\coloneqq	d_{\wm}(\widetilde{\stheta}, \tau \cdot \widetilde{\stheta}). 
	$$
		The \textbf{fiber translation length} of $f$ is defined by
		$$ d(f)\coloneqq \sup\limits_{\theta\in S^1} d(\stheta).$$
		The \textbf{fiber translation length} of $\alpha$ is
		\[ d_{\alpha} \coloneqq \sup_{\substack{f\in \alpha }} d(f). \]
	\end{defi}
\noindent The fiber translation length $\da$ is equal to 
	\[	 \sup_{\substack{S\subset M \tn{ isotopic} \\ \tn{to a fiber}}}	d_{\wm}(\widetilde{S}, \tau \cdot \widetilde{S}).\]

\fibertranslationLengthequalReciprocalK*
	\begin{proof}
		Let $f\in \alpha$ be tight. We first show that any geodesic segment $\widetilde{\eta}\subset \Lambda(f)$ between $\widetilde{S}_0$ and $\widetilde{S}_1 = \tau \cdot \widetilde{S}_0$ has a length larger than or equal to $1/K$, with equality only if it belongs to the maximal stretch lamination. As in the proof of \Cref{level set of tight cut lamination equally}, 
		$$
				1=\int_{\widetilde{\eta}} d\widetilde{f} = \int_{0}^{\ell(\widetilde{\eta})} \widetilde{\eta}'(t)\cdot \nabla \widetilde{f}(\widetilde{\eta}(t)) \di t \leq \int_{0}^{\ell(\widetilde{\eta})} |\nabla \widetilde{f}(\widetilde{\eta}(t))| \di t \leq \ell(\widetilde{\eta}) \cdot K,
		$$
		where the inequality is equal if $\widetilde{f}$ restricted to $\widetilde{\eta}$ has the largest Lipschitz constant $K$: $|\nabla \widetilde{f}(\widetilde{\eta}(t))|=K$, which implies that $\widetilde{\eta}$ is a subset of the maximal stretch locus. Thus any geodesic segment between $\widetilde{S}_0$ and its translate other than \msl has length strictly larger than $1/K$. \par
		For a general $f\in \alpha$, we will show $d(f) \leq \frac{1}{K}$. 
		Let $\epsilon>0$ be small and $\gamma_n$ a closed geodesic such that $K-\epsilon < K(\gamma_n) $ and $\alpha(\gamma_n)=n$. In $\wm$, a lift $\widetilde{S}$ and its translate $\tau^n(\widetilde{S})$ bound a fundamental domain $D_n$ for the $n$-fold cover of $M$. Since $\tau$ maps one fundamental domain of $M$ to another as an isometry, 
		$$d(\tau^{i}\cdot \widetilde{S}, \tau^{i+1}\cdot \widetilde{S})= d(\widetilde{S}, \tau\cdot \widetilde{S}). $$
		 Denote the intersection of the lift $\widetilde{\gamma_n}$ and $\tau^i \cdot \widetilde{S}$ by $p_i$. We have 
		\begin{align*}
			n\cdot d(\widetilde{S}, \tau \cdot \widetilde{S}) &= \sum\limits_{ i=0}^{n-1} d(\tau^{i}\cdot \widetilde{S},   \tau^{i+1}\cdot \widetilde{S}) \\
			&\leq \sum\limits_{ i=0}^{n-1} d(p_i, p_{i+1}) = \ell(\gamma_n),
		\end{align*}
		which implies \[ d(\widetilde{S}, \tau \cdot \widetilde{S}) \leq \frac{ \ell(\gamma_n)}{n}=\frac{ \ell(\gamma_n)}{\alpha(\gamma_n)} =\frac{1}{K(\gamma_n)}\leq \frac{1}{K-\epsilon}.\]
		Since $\epsilon$ is arbitrary, we have 
		\[ d(\widetilde{S}, \tau \cdot \widetilde{S}) \leq \frac{1}{K}\]
	\end{proof}

	Let $\pi_n: S^1 \rightarrow S^1$ be the degree-$n$ covering map. The behavior of the Lipschitz constant under covering maps is used to demonstrate the sharpness of \Cref{inequality Lipschitz constant entropy} in \Cref{lem product Lipschitz entropy}. 
\begin{lem}\label{lipschitz constant under covering}
Let $f$ be continuous and $M_n$ the degree-$n$ cover of $M$. Assume the maps $f_n \coloneqq f\circ\pi_n$ and $f_{1/n}$ make the following diagrams commute
\begin{figure}[ht]
	\centering
\subfloat[$f_n$]{
	\includegraphics[scale=0.5]{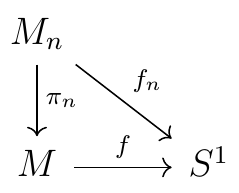} 
	\label{fovern}
} \hspace{0.5in}
\subfloat[$f_{1/n}$]{
		\includegraphics[scale=0.5]{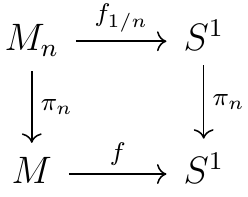}.
		\label{fn}
}
\end{figure}

Then $$
	\lip_{f_n}=\lip_{f} \tn{ and } \lip_{f_{1/n}}=\frac{\lip_{f}}{n}.
$$
\end{lem}
\begin{proof}
	The map $\pi_n$ is a local isometry. Take $y\in \pi_n^{-1}(x)$. Then with the definition of the Lipschitz constant as a ratio of distances in \Cref{def Lipschitz constant of homotopy},  
	$$ \lip_{f_n}(y) = \lim\limits_{r\rightarrow 0} \lip_{f_n}(B_r(y))= \lim\limits_{r\rightarrow 0} \lip_{f}(B_r(x))=\lip_f(x).$$
	Thus $\pi_n$ pullbacks the \msl of $f$ to the \msl of $f_n$, with the same Lipschitz constant. This finishes the proof of the first equality. \par  
	The differential $d\pi_n$ of the map $\pi_n: S^1 \rightarrow S^1$ is given by the multiplication by $n$: $d \pi_n (v) =nv$. Note that $f_n=\pi_n \circ f_{1/n}$.  It implies that for $y_1, y_2 \in M_n$ very close, 
	$$d_{S^1}(f_n(y_1), f_n(y_2)) = d_{S^1}(\pi_n\circ f_{1/n}(y_1), \pi_n\circ f_{1/n}(y_2))=n\cdot d_{S^1}(f_{1/n}(y_1), f_{1/n}(y_2)) .$$
	From $d_{S^1}(f_{1/n}(y_1), f_{1/n}(y_2))=\frac{1}{n}d_{S^1}(f_n(y_1), f_n(y_2))$, we have 
	$$ \lip_{f_{1/n}}(y) = \lim\limits_{r\rightarrow 0} \lip_{f_{1/n}}(B_r(y))= \frac{1}{n}\lim\limits_{r\rightarrow 0} \lip_{f_n}(B_r(y))=\frac{1}{n}\lip_{f_n}(y)=\frac{1}{n}\lip_{f}(x).$$

\end{proof}
\noindent Taking the degree-$n$ cover corresponds to truncating $\widetilde{\Lambda(f)}$ by the $n$-fold fundamental domain of $M$ in $\wm$. \Cref{lipschitz constant under covering} is known to Farre, Landesberg, and Minsky \cite{flmMinimizingLaminationsRegularCover} since they describe the \msl as the fastest routes traversing the \zcover and taking a finite cover does not change the associated \zcover. \par
 A chain recurrent sublamination $\Lambda^{cr}(\alpha) \subset \Lambda(\alpha)$ that is maximal with respect to inclusion is called the cohomology stretch lamination by Rudd \cite{rcStretchLamination}. When $\dim M=2$, a transverse measure $m$ with support on a maximal measured sublamination $\Lambda_m(\alpha)\subset \Lambda^{cr}(\alpha)$ exists by [\citenum{dubestLipLG}, Corollary 6.8]. The lamination $\Lambda(\alpha)$ is not only oriented, it has an orientation so that a submanifold $S$ dual to $\alpha$ exists with all intersections $\Lambda(\alpha) \cap S$ positive. For hyperbolic surfaces, let $\langle\cdot,\cdot\rangle$ denote the intersection forms for geodesic currents. 
\begin{prop}\label{algebraic intersection=geometric intersection}
Let $M^n$ ($n\geq 2$) be a closed hyperbolic manifold and $\alpha\in H^1(M;\Z)$.  There exists an oriented submanifold $S$ dual to $\alpha$ and an orientation on $\Lambda(\alpha)$ such that all intersections $\Lambda(\alpha) \cap S$ have positive signs. For hyperbolic surfaces, a multi-curve $\beta_{\alpha}$ dual to $\alpha$ exists so that 
	\begin{equation}\label{eq algebraic intersection=geometric intersection}
	\langle \Lambda_m(\alpha), \beta_{\alpha}\rangle=K\ell(\Lambda_m(\alpha)).
 	\end{equation}
\end{prop}
\begin{proof}
	\cite{fsExistenceCriticalSubsolutions} shows that there exists a $C^1$ tight map in the homotopy class of $f$. Assume $f$ is $C^1$ tight and we will show the \msl $\Lambda(f)$ of $f$ satisfies \eqref{eq algebraic intersection=geometric intersection}. Let $\gamma \subset \Lambda(f)$ be a geodesic leaf parameterized such that $\gamma'(t) = \frac{1}{K}\nabla f(\gamma(t))$. Since both $M$ and $S^1$ are compact, by the classical smooth approximation theorem, for any $\epsilon>0$ there exists a smooth $h: M\rightarrow S^1$ homotopic to $f$ such that $\|h-f\|_{C^1} < \epsilon$, which means 
	\[\|h(x)-f(x)\|+\|\nabla h(x)- \nabla f(x)\|<\epsilon.\]
	By Sard's theorem, regular values of $h$ are of full measure on $S^1$. By the \pc duality, the preimage $S$ of any regular value  $\theta$ of $h$ is dual to $\alpha$. By the compactness of $M$, $S$ is compact. For the simplicity of notations, assume $S$ is a component which $\gamma$ intersects at a point $p$. Without loss of generality, a frame $F_p$ for $T_pS$ and $\nabla h$ are positively oriented (and orthogonal). Since $ \|\nabla h(x)- \nabla f(x)\|<\epsilon$ for a small $\epsilon$, we have $\{F_p,\nabla f(p)\}=\{F_p, \gamma'(p)\}$ is also positively oriented (and almost orthogonal). This applies to all intersections of $\gamma \cap S$. Thus we have proven for every component $S$ of $h^{-1}(\theta)$, the intersection $\gamma \cap S$ has a consistent sign. This argument generalizes to other leaves of $\Lambda(f)$, which shows that $\Lambda(f)$ at every point of $h^{-1}(\theta)\cap \Lambda(f)$ can be consistently oriented as $\nabla h$ so that all algebraic intersections are positive. \par
	Now suppose $M$ is a hyperbolic surface. We first assume the measured lamination $\Lambda_m(\alpha)$ consists of multi-geodesics. Let $\beta_{\alpha}$ be a regular level set of $h$ which is a $1$-cycle dual to $\alpha$. Since all intersections can be oriented positively, we have 
	\begin{equation}\label{eq geometric intersection=Klength}
		\langle \Lambda_m(\alpha), \beta_{\alpha} \rangle = \int_{\Lambda_m(\alpha)} \alpha = K \ell(\Lambda_m(\alpha)).
	\end{equation} 
	Since weighted closed geodesics are dense in the space of measured laminations ([\citenum{phCombinatoricsTrainTrack}, Theorem 3.1.3]), and the length functional is continuous ([\citenum{mbIntroGeometricTopology}, Theorem 8.2.21, 8.3.6 (3)]), \eqref{eq geometric intersection=Klength} holds for measured lamination as well. 
	\end{proof}
Using \eqref{eq algebraic intersection=geometric intersection}, we can define the length-$1$ geometric intersection between the measured lamination $\Lambda_m(\alpha)$ and the submanifolds $S$ which are regular level sets of $h$ in the following way.
\begin{defi}\label{defi geometric=K times length}
	The \textbf{length-$1$ geometric intersection number} of $S$ and the geodesic lamination $\Lambda_m(\alpha)$ is defined by
	$$
		\langle \Lambda_m(\alpha), S \rangle = K.
$$
\end{defi}
\noindent For hyperbolic surfaces, this corresponds to the transverse measure of a length-$1$ lamination in \PML.  [\citenum{flmMinimizingLaminationsRegularCover}, Theorem 1.7]\footnote{They have normalized so that the Lipschitz constant $K=1$. One of the merits of their normalization is that $\tilde{f}$ maps any leaf of $\widetilde{\Lambda(f)}$ isometrically into $\R$.} shows on a closed surface $M$, any orientable measured lamination can occur as the maximal stretch lamination for some suitable hyperbolic metric and tight map. This means that on a surface, the set of laminations that \Cref{defi geometric=K times length} holds includes all orientable geodesic laminations, which is compatible with the intersection form. This also generalizes the ordinary geometric intersection number if $\Lambda_m(\alpha)$ consists of closed geodesic (up to the length). 
 \subsection{The Thurston norm, the entropy, and the fiber translation length}
 The Thurston norm is a topological measurement of the level sets of a circle-valued map $f\in \alpha$. The Lipschitz constant is a geometric measurement of $\nabla f$ orthogonal to the level sets. They are related by
 \thurstonnormboundedbyLipschitzconstant* 
\begin{proof}
	Let $f\in \alpha$ be a $C^1$ tight map and thus $\|df\|_{\li}=\|[\alpha]\|_{\li}$. Let $h\in \alpha$ such that $h$ is smooth and $\|h - f \|_{C^1} < \epsilon$ for a small $\epsilon >0$.
	[\citenum{hansHarmonicForms}, Proposition 5.2.1] proves $\pi \|\alpha\|_{Th} < \|\alpha\|_{LA}$ and [\citenum{bdNorms}, Lemma 3.1] proves that least area norm is equal to the $L^1$-norm $$\|\alpha\|_{LA} = \|\alpha\|_{L^1}.$$ Since $\|\alpha\|_{L^1}$ is the minimal $L^1$-norm and $[dh] = \alpha$, we have $$\|\alpha\|_{L^1} \leq \int_M |dh|. $$
	 Now	 
	\[
	\pi\|\alpha\|_{Th} < \|\alpha\|_{LA} = \|\alpha\|_{L^1}   \leq \int_M |dh| \leq \int_M (|df|+\epsilon)\leq \vol(M)(\|df\|_{\li}+\epsilon).
	\]
	Since $\epsilon$ is arbitrary, we have $$\pi\|\alpha\|_{Th}< \vol(M)\|df\|_{\li}=\vol(M) \lip_{[\alpha]}.$$
\end{proof} 
As a corollary, we can prove a volume comparison result for a Seifert-fibered $3$-manifold and a fibered hyperbolic $3$-manifold assuming the Lipschitz constant of the fibrations are the same. 
\volTrivialMappingtorusLessthanHyperbolic*
\begin{proof}
Recall $S^1=\R/Z$ has a standard metric of length $1$ and $M_{Id}=S_g \times S^1$ is equipped with a product metric. Define $f_1:M_{Id}\rightarrow S^1=\R/\Z$ be the map defined by $(x, t) \rightarrow \lip_{[f]}\cdot t \in S^1$ which uniformly stretches every vertical geodesic by $\lip_{[f]}$. The map $f_1$ is a fibration which actually minimizes the Lipschitz constant in its homotopy class $[f_1]$ by an argument similar to the proof of \Cref{fiber translation Length equal Reciprocal K}. Let $\gamma$ be a vertical closed geodesic, whose length is $1/\lip_{[f_1]}$. 
	The volume of $M_{Id} = \area(S_g) \ell(\gamma) = 2\pi\chi(S_g) \ell(\gamma)= 2\pi\chi(S_g)/\lip_{[f_1]}=2\pi\|df\|_{Th}/\lip_{[f]} < 2\vol(M_{\phi})$. 
\end{proof}
 Let $S$ be a hyperbolic surface of genus $g$, and $\phi: S \rightarrow S $ a pseudo-Anosov homeomorphism. Denote the entropy of $\phi$ by $\text{ent}_{\phi}$, which is the infimum of the topological entropy of automorphisms of $S$ isotopic to $\phi$ [\citenum{flp}, Proposition 10.13]. Let $M\coloneqq M_{\phi}$ be the mapping torus $S \times [0,1] / (x,1) \sim (\phi(x), 0)$. The topological type of $M$ only depends on the isotopy class of $\phi$. A celebrated theorem by Thurston \cite{thurston1998hyperbolic} asserts that $M$ is a closed hyperbolic $3$-manifold. The projection of $S \times [0,1]$ onto the second factor induces a map from $M$ to the circle which is a fibration $f: M \rightarrow S^1$ with fiber $S$. The \zcover $\widetilde{M}$ of $M$ given by this fibration is homeomorphic to $S \times \R$ with deck transformations generated by $\tau(x, t) = (\phi(x), t+1)$. Let $K = \lip_{[f]}$ be the smallest Lipschitz constant associated with the homotopy class $[f]$. 
Let us now prove
\inequalityLipschitzConstantEntropy*
\begin{proof}
	Since $S$ is homotopic to a fiber of $f$, by [\citenum{thurstonNormHomology3manifold}, Theorem], we have the Thurston norm $\|df \|_{Th} = |\chi(S)|$. By [\citenum{kmNormalizedEntropyVolumePA}, Theorem 1.1], the volume of $M$ satisfies
	\[\vol (M) \leq 3\pi |\chi(S)| \cdot \text{ent}_{\phi}.\]
	Thus by \eqref{eq thurston norm bounded by Lipschitz constant} we have \[\pi |\chi(S)| < 3K \pi |\chi(S)| \cdot \text{ent}_{\phi},\] which implies
	\[1 < 3 K \cdot \text{ent}_{\phi}.\]
\end{proof}

\productLipschitzconstantentropy*
\begin{proof} 
	Let $f:M=S \times [0,1] / (x,1) \sim (\phi(x), 0) \rightarrow S^1$ be tight. Let $M_n= S \times [0,1] / (x,1) \sim (\phi^n(x), 0)$ be a degree-$n$ cover of $M$. Then the product of the Lipschitz constant of $f_{1/n}$ and the entropy stays invariant. This follows from \Cref{lipschitz constant under covering} and the formula
	\[ \tn{ent}_n=\textnormal{ent}_{\phi^n}=n \cdot \textnormal{ent}_{\phi}\]
	from [\citen{flp}, Proposition 10.4]. 
\end{proof}

It is an open question whether there exists a best Lipschitz map $h$ in the homotopy class of $f$ which induces a fibration. If such a fibration exists, we can extract from the proof above that every fiber of $h$ is equally translated by $\tau$ in $\wm$ by $\frac{1}{K}$. 
\appendix
\section{Harmonic expansion in a Margulis tube}\label{appendix}
Harmonic functions on Euclidean space in cylindrical coordinates or on hyperbolic space in spherical coordinates are well-studied. To the best knowledge of the author, it is difficult to locate references on harmonic functions on Margulis tubes in terms of hypergeometric functions, except \cite{vkHarmonicHyperbolicCylindrical}. 

\subsection{Formulas, the equation, and the series expansion in terms of hypergeometric functions}
Denote the length of its core geodesic $\gamma$ by $\lambda$ and the radius of a tube $\T$ by $R$. Let $z \in [0, \lambda]$ be the unit-speed parameterization of the geodesic. Let $r \in [0, R]$ be the radial coordinate and $\theta$ be the angular coordinate. The coordinates come with identification $(r, \theta, \lambda)\sim (r, \theta +\theta_0, 0)$ where $\theta_0$ is the twist parameter determined by the geometry of $\T$. The \textbf{metric} on $\T$ is then given by 
$$
	g=dr^2+\sinh^2{r} \di \theta^2 + \cosh^2{r} \di z^2.
$$
The pointwise lengths are
$$
	|dr|=1, |d\theta| =\frac{1}{\sinh r }, \tn{ and } |dz| =\frac{1}{\cosh r}. 
$$
An \textbf{orthonormal basis} of $1$-forms is 
$$
	dr, \, \sinh r \, d\theta, \cosh r\, dz.
$$
The \textbf{volume of the tube} $\T$ is 
\begin{equation}\label{eq vol tube}
	\ds \vol(\T)=\int_{0}^{\lambda}\int_{0}^{R}\int_{0}^{2\pi}dr \wedge \sinh r \di \theta \wedge \cosh r\,dz =\pi\lambda\sinh^2R. 
\end{equation}
The \textbf{area of a totally geodesic disk} $D_r =\{ z=constant\}$ is
$$
	\area(D)= 2\pi (\cosh r-1).
$$
Let $f(r,\theta, z)$ be a harmonic function defined on the tube $\T$. We have
\[df=f_rdr+f_{\theta}\di \theta+f_zdz\]
and 
$$
	*df=f_r\sinh r \cosh r \di \theta\wedge dz+ f_{\theta}\frac{\cosh r}{\sinh r} dz\wedge dr + f_z\frac{\sinh r}{\cosh r} dr\wedge \di \theta.
$$
The \textcolor{blue}{harmonic equation} for $f$ is
$$
	\Delta f =*\di *df= \partial_r(f_r \sinh r \cosh r) +\partial_{\theta}f_{\theta}\frac{\cosh r}{\sinh r}+ \partial_z f_z\frac{\sinh r}{\cosh r}=0.
$$
If we use separation of variables, $f(r, \theta, z)=h(r)s(\theta, z)$, we have 
$$
	\frac{1}{2} \partial_r(h'(r)\sinh 2r ) s(\theta, z) + \frac{\cosh r}{\sinh r}h(r)\frac{\partial^2 s(\theta, z)}{\partial \theta^2} +h(r)\frac{\sinh r}{\cosh r}\frac{\partial^2 s(\theta, z)}{\partial z^2}=0. 
$$
Then 
\begin{equation}\label{harmonicSeparationVariable}
	\frac{1}{2}\frac{\partial_r(h'(r)\sinh 2r )}{h(r)}+\frac{\cosh r }{\sinh r}\frac{\partial^2 s(\theta, z)}{\partial \theta^2}\frac{1}{s(\theta, z)}+\frac{\sinh r }{\cosh r }\frac{\partial^2 s(\theta, z)}{\partial z^2}\frac{1}{s(\theta, z)}=0.	
\end{equation}
The function $f$ satisfies certain periodicity with respect to $\theta$ and $z$: 
\begin{equation}\label{periodTheta}
	f(\theta)=f(\theta+2\pi),
\end{equation}
and
\begin{equation}\label{periodZ}
	f(r, \theta+\theta_0, z)= f(r, \theta, z+\lambda). 
\end{equation}
Differentiate \eqref{harmonicSeparationVariable} with respect to $\theta$, we get
\[\frac{\partial^2 s_{km}(\theta, z)}{\partial \theta^2} = -k^2  s_{km} \text{ for } k \in \nz, \]
where we have used \eqref{periodTheta}. 
Differentiate \eqref{harmonicSeparationVariable} with respect to $z$, we get
\[\frac{\partial^2 s_{km}(\theta, z)}{\partial z^2} = -\left(\frac{2m\pi}{\lambda}+\frac{k\theta_0}{\lambda} \right)^2  s_{km} \text{ for } m \in \nz, \]
using \eqref{periodZ}. 

\begin{lem}[Laplacian equation in cylindrical coordinates and the separation of variables]
	Any harmonic function inside the tube can be expressed in terms of the series
	$$
		\ds	f(r,\theta, z)=\sum\limits_{k,m=0}^{\infty} h_{km}(r) s_{km}(\theta, z),
	$$
	where $h_{km}(r)$ is the solution to 
	\begin{equation}\label{eq radial part of harmonic}
		\frac{1}{2}\frac{\partial_r(h'(r)\sinh 2r )}{h(r)}-k^2\frac{\cosh r }{\sinh r} - \left(\frac{2m\pi}{\lambda}+\frac{k\theta_0}{\lambda}\right)^2 \frac{\sinh r }{\cosh r }=0.
	\end{equation}
\end{lem}
Here $\int_{0}^{2\pi }\skm(\theta) d\theta=0$ unless $k=0$. We first analyze $h_{00}(r)$ which satisfies \par  
\[\partial_r(h'(r)\sinh 2r) =0.\] We have
\[
h(r)=c_1 + c_2 \log \frac{\sinh r}{\cosh r}.
\]
Since $h(r)$ is well-defined at $r=0$, $c_2$ has to be $0$. Thus $h_{00}(r) \equiv c_1$. This is essentially what Brock and Dunfield show in [\citenum{bdNorms}, p. 549] that harmonic $1$-forms independent of $\theta$ and $z$ coordinates are proportional to $dz$. We will henceforward assume either $m$ or $k \geq 1$.  

For general $m$ with $k=0$, the radial equation is
\begin{equation}\label{harmonic equation without k}
	\partial_r(h'(r)\sinh 2r )- 2h(r) \frac{4\pi^2m^2}{\lambda^2} \frac{\sinh r }{\cosh r}=0.
\end{equation} 
For general $m, k$, the equation is 
$$
	h''(r)\sinh 2r+ h'(r)2\cosh 2r - 2k^2 h(r) \frac{\cosh r}{\sinh r}- 2\left(\frac{2m\pi}{\lambda}+\frac{k\theta_0}{\lambda}\right)^2h(r) \frac{\sinh r }{\cosh r}=0.
$$
For $r \neq 0$, dividing by $\sinh 2r$, we have 
\[h''(r)+ 2h'(r)\frac{\cosh 2r}{\sinh 2r} - k^2 h(r) \frac{1}{\sinh^2 r}- \left(\frac{2m\pi}{\lambda}+\frac{k\theta_0}{\lambda} \right)^2 h(r) \frac{1}{\cosh^2 r}=0\]
After the change of variable $u=(\tanh r)^2$, using
\begin{align*}
\frac{du}{dr}&=2\tanh r\frac{1}{\cosh^2 r}=2\sqrt{u}(1-u), 	\frac{1}{\sinh^2 r} = \frac{1-u}{u},\frac{1}{\cosh^2 r} =1-u \\
h''(r) &= \frac{du}{dr}\frac{d}{du}(2h'(u)\sqrt{u}(1-u)) =4h''(u)u(1-u)^2 + 2h'(u)(1-3u)(1-u)\\
\cosh 2r &= \cosh^2 r +\sinh^2 r =\frac{1+u}{1-u}, \, \frac{\cosh 2r}{\sinh 2r} =\frac{u+1}{2\sqrt{u}}
\end{align*}
we have
\begin{align*}
&4h''(u)u(1-u)^2+2h'(u)(1-3u)(1-u)+ 2h'(u)(1+u)(1-u) \\ 
&- k^2h(u)\frac{1-u}{u}- \left(\frac{2m\pi}{\lambda}+\frac{k\theta_0}{\lambda} \right)^2 h(u)(1-u)=0,
\end{align*}
which is equivalent to
\[(1-u)\left(h''(u)u(1-u)+h'(u)(1-u) -\frac{k^2}{4}h(u)\frac{1}{u} -(\frac{m\pi}{\lambda}+\frac{k\theta_0}{2\lambda})^2 h(u)\right)=0.\]
For $r \in [0, \infty)$, $u \in [0, 1)$. Thus we can factor out $1-u$ and obtain
\begin{equation}\label{kmHarmonicWRTu}
	h''(u)u(1-u)+h'(u)(1-u) -\frac{k^2}{4}h(u)\frac{1}{u} -(\frac{m\pi}{\lambda}+\frac{k\theta_0}{2\lambda})^2h(u)=0
\end{equation}
Divide by $u(1-u)$, we have 
$$
	h''(u)+h'(u)\frac{1}{u} -\frac{k^2}{4}h(u)\frac{1}{u^2(1-u)} -(\frac{m\pi}{\lambda}+\frac{k\theta_0}{2\lambda})^2 h(u)\frac{1}{u(1-u)}=0.
$$
Set $h(u)=u^{\frac{k}{2}} \phi(u)$, then using the formula after [\citenum{specialFunctionsgrr}, (2.3.4)] and 
\[\frac{-k^2}{4}\frac{1}{u^2(1-u)}+\frac{k^2}{4}\frac{1/u}{u}=-\frac{k^2}{4}\frac{1}{u(1-u)},\]

we have that $\phi(u)$ satisfies
$$
	\phi''(u)+\frac{1}{u}(1+k)\phi'(u)-\frac{1}{u(1-u)}(\frac{k^2}{4}+(\frac{m\pi}{\lambda}+\frac{k\theta_0}{2\lambda})^2)\phi(u)=0.
$$
Compared with [\citenum{specialFunctionsgrr}, (2.3.5)], first assume $k\neq 0$. 
There are two solutions,
\begin{align*}
	\phi_1(u) &= \tfo\left(\frac{k}{2} +i(\frac{m\pi}{\lambda}+\frac{k\theta_0}{2\lambda}), \frac{k}{2} -i(\frac{m\pi}{\lambda}+\frac{k\theta_0}{2\lambda}); 1+ k;u \right)\\
	&\tn{and} \\
	\phi_2(u) &=u^{-k} \tfo\left(-\frac{k}{2} +i(\frac{m\pi}{\lambda}+\frac{k\theta_0}{2\lambda}), -\frac{k}{2} -i(\frac{m\pi}{\lambda}+\frac{k\theta_0}{2\lambda});1-k;u \right) \\
\end{align*}
by [\citenum{specialFunctionsgrr}, Remark 2.3.1]. But $h(u)=u^{\frac{k}{2}} \phi_2(u)$ is not well-defined at $u=0$ by [\citenum{specialFunctionsgrr}, (1.1.2), (2.1.2) the definition of a hypergeometric function]. If $k=0$, in addition to $\phi_1(u)$, the other independent solution [\citenum{specialFunctionsgrr}, (2.3.18)] contains a term of $\log u$ which is not well-defined at $u=(\tanh 0)^2=0$. Thus we have 
\begin{lem}
	The solution to the hypergeometric equation \eqref{kmHarmonicWRTu} for general $m, k$ which is well-defined at $u=0$ is
	$$
		h_{km}(u)= u^{\frac{k}{2}} \tfo\left(\frac{k}{2} +i(\frac{m\pi}{\lambda}+\frac{k\theta_0}{2\lambda}), \frac{k}{2} -i(\frac{m\pi}{\lambda}+\frac{k\theta_0}{2\lambda});1+k;u \right),
$$
	and in terms of the original variable $r$, 
	$$
		h_{km}(r)= \tanh^k r \, \tfo\left(\frac{k}{2} +i(\frac{m\pi}{\lambda}+\frac{k\theta_0}{2\lambda}), \frac{k}{2} -i(\frac{m\pi}{\lambda}+\frac{k\theta_0}{2\lambda});1+k;(\tanh r)^2 \right). 
	$$
\end{lem}

The radius of convergence of a hypergeometric function $_2F_1(\,,\,;\,;u)$ is $|u|<1$. The functions $h_{km}(r)$ are strictly increasing by the positivity of the coefficients in its Taylor expansion.
 By [\citenum{specialFunctionsgrr}, Theorem 2.1.2], since $1+k-(\frac{k}{2}+i(\frac{m\pi}{\lambda}+\frac{k\theta_0}{2\lambda})+\frac{k}{2}-i(\frac{m\pi}{\lambda}+\frac{k\theta_0}{2\lambda}))=1>0$, the asymptotic as $r \rightarrow \infty$ of $h_{km}(r)$ for $k \geq 0$ is finite and will be denoted as $h_{km}(\infty)$. We can also check that for $(k, m) \neq (0,0)$, 
 \begin{equation}\label{skmThetaz}
 	s_{km}(\theta, z)=a_{km} \sin(k\theta+\frac{2\pi m}{\lambda}z+\frac{k\theta_0}{\lambda}z)+a_{km} '\cos(k\theta+\frac{2\pi m}{\lambda}z+\frac{k\theta_0}{\lambda}z).
 \end{equation}
 
\begin{theorem}\label{theorem harmonic expansion into hypergeometric}
		Any harmonic function well-defined in the tube with core length $\lambda$ and twist parameter $\theta_0$ can be expressed in terms of the series
	\begin{align*}
\numberthis \label{harmonic expansion into hypergeometric}		\ds	f(r,\theta, z)&=\sum\limits_{k,m=0}^{\infty} \tanh^k r \, \tfo\left(\frac{k}{2} +i(\frac{m\pi}{\lambda}+\frac{k\theta_0}{2\lambda}), \frac{k}{2} -i(\frac{m\pi}{\lambda}+\frac{k\theta_0}{2\lambda});1+k;(\tanh r)^2 \right) \\ &\cdot\left(a_{km} \sin(k\theta+\frac{2\pi m}{\lambda}z+\frac{k\theta_0}{\lambda}z)+a_{km} '\cos(k\theta+\frac{2\pi m}{\lambda}z+\frac{k\theta_0}{\lambda}z)\right).
	\end{align*}
\end{theorem}
\subsection{The conjecture without volume upper bound}\label{counter example Brock-Dunfield without GH convergence}
\Cref{integral *df least area close to totally geodesic} implies that asymptotically for $\int_{D}*df$ there is little difference between a least area $D$ and a totally geodesic one. We will construct explicit examples using totally geodesic disks, which significantly simplifies the computation. On such a geodesic disk, for the hypergeometric expansions we only need to consider terms with $k=0$, since $$ \int_{D}*df=\int_{0}^{R}\int_{0}^{2\pi}f_z \tanh r \di\theta \wedge dz$$ and $\int_{0}^{2\pi}s_{km}(\theta, z)d\theta=0$ when $k \neq 0$. With the change of variable $u=(\tanh r)^2$, the radial part of the harmonic equation \eqref{eq radial part of harmonic} becomes  
$$
	h_m''(u)u(1-u)+h_m'(u)(1-u)-\frac{\pi^2 m^2}{\lambda^2}h_m(u)=0.
$$
The solution well-defined at $0$ is 
$$
	h_m(u)= \tfo\left(i\frac{\pi m}{\lambda},-i\frac{\pi m}{\lambda};1;u \right).
$$
 By \eqref{harmonic equation without k},
\begin{align*}\label{integral*dfTotallyGeodesic}
	\int_{0}^{R} h_m(r) \frac{\sinh r }{\cosh r} dr &
	=\frac{\lambda^2}{8\pi^2m^2}h_m'(R) \sinh 2R \\
	&= \frac{\lambda^2}{4\pi^2m^2} h_m'(u)\sqrt{u}(1-u)\sinh 2R = \frac{\lambda^2}{2\pi^2m^2} h_m'(u)u. \numberthis
\end{align*}
Here $u=(\tanh R)^2$ and $\vol(\T)=\pi\lambda \sinh^2 R$ is finite.  
By [\citenum{specialFunctionsgrr}, (2.5.1)], 
$$
	h_m'(u)= \frac{\pi^2 m^2}{\lambda^2}\,\tfo\left(1+ i\frac{\pi m}{\lambda}, 1-i\frac{\pi m}{\lambda}; 2 ; u \right)
$$
which asymptotic at $u=1$ is described by Gauss. We obtain an inequality which complements Gauss' asymptotic on $\lim\limits_{r\rightarrow \infty}h_m'(r)$.
\begin{lem}
	For any $d>0$, the following inequality holds for hypergeometric functions.
	\begin{equation}\label{hypergeometric(ai+ -ai+ c+)<hypergeometric(ai -ai c) log 1/(1-u)}
		\tfo(1+ di,1-di; 2; u) u < \tfo( di,-di; 1; u) \log \frac{1}{1-u}.
	\end{equation}
	Equivalently, for $h_m(r)=\tfo(i \frac{\pi m}{\lambda},-i \frac{\pi m}{\lambda}; 1; (\tanh r)^2)$, 
	\begin{equation}\label{h'(r)< h(r)8d^2 logcoshr /sinh^2r}
		h_m'(r)   < 4(\frac{\pi m}{\lambda})^2 \frac{h_m(r) \log \cosh r}{\sinh r \cosh r}. 
	\end{equation} 
\end{lem}
\noindent\eqref{hypergeometric(ai+ -ai+ c+)<hypergeometric(ai -ai c) log 1/(1-u)} and \eqref{h'(r)< h(r)8d^2 logcoshr /sinh^2r} are asymptotically sharp by Gauss' asymptotic for hypergeometric functions [\citenum{specialFunctionsgrr}, Theorem 2.1.3], i.e.,  
\begin{equation}\label{limit h'R}
	\lim\limits_{r\rightarrow \infty}h_m'(r) =\lim\limits_{r\rightarrow \infty}4(\frac{\pi m}{\lambda})^2 \frac{h_m(r) \log \cosh r}{\sinh r \cosh r}
\end{equation}
\begin{proof}
	Let $d=\frac{\pi m}{\lambda}$, 
	\[\int_{0}^{R}h_m(r)\frac{\sinh r}{\cosh r}dr= \frac{1}{4d^2}h'_m(R)\sinh R\cosh R= \frac{1}{2d^2}h_m'(u)u,\] where we have used $\frac{du}{dr}= 2\sqrt{u}(1-u)=2\tanh r \frac{1}{\cosh^2 r}$. Now
	\begin{align*}
		&\int_{0}^{R}h_m(r)\frac{\sinh r}{\cosh r} dr < h_m(R)\int_{0}^{R}\frac{\sinh r}{\cosh r}dr\\
		=&h_m(R) \log \cosh R = \frac{1}{2} h_m(u) \log \frac{1}{1-u}.
	\end{align*}
\end{proof}

Combining \eqref{integral*dfTotallyGeodesic} and \eqref{limit h'R}, we have
$$
	\int_{0}^{R} h_m(r) \frac{\sinh r }{\cosh r} dr = b_m \tanh^2 R \, \log\cosh R. 
$$
We also need a basic lemma concerning the radius and injectivity radius of a Margulis tube. 
\begin{lem}\label{lambda log cosh R}
	Suppose $\T$ is a Margulis tube with core geodesic length $\lambda$, radius $R$, and Margulis constant $0.29$. Then as $\lambda \rightarrow 0$, 
	\begin{equation}
		\lambda \log (-\lambda) (\log \cosh R) \rightarrow 0. 
	\end{equation}
\end{lem}
\begin{proof}
	We use [\citenum{fpsEffectiveDistanceNestedTubes}, Proposition 5.7] to bound the radius of a tube from above. Using their notation, $\epsilon=0.29$ is the Margulis constant and $\delta=\lambda$ is the length of the core geodesic. Since $\T^{\leq \lambda} \subset \T$ is the core geodesic, $R=d_{2\pi, \lambda, \tau}$ is the radius. We have 
	$$\cosh R \leq \sqrt{\frac{\cosh 0.29-1}{\cosh \lambda -1}},$$
	which implies 
	$$\log \cosh R\leq \frac{1}{2}\log \frac{0.043}{\cosh \lambda -1}. $$
	Using L'Hopital's rule, it is easy to see that 
	$$\lim\limits_{\lambda\rightarrow 0} \lambda \log (-\lambda) \log \frac{0.043}{\cosh \lambda -1}  \rightarrow 0. $$	
\end{proof}

We are now ready to show
\begin{theorem}\label{local conjecture l infinity not bounded by sqrt inj l2}
Let $M$ be a closed hyperbolic $3$-manifold with $b_1(M)>0$. Then there \textbf{does not} exist a uniform constant $c$ such that for harmonic $1$-forms $\alpha$,
	\[\|\alpha\|_{\li} \leq c\sqrt{-\log \inj(M)}\|\alpha\|_{L^2}.\] 
\end{theorem}
\begin{proof}
	For $[dz]\in H^1(\T_R; \R)$ dual to $[D, \partial D] \in H_2(\T_R, \partial \T; \R)$, we have 
	\[ \|dz\|_{\li(\T)}= \max\limits_{r\in [0, R]}\frac{1}{\cosh r}=1 \]
	and \[ \|dz\|_{L^2(\T)} = \sqrt{2\pi \lambda\log \cosh R}. \]
	Thus by \Cref{lambda log cosh R}, for any $c>0$, 
	\[\|dz\|_{\li(\T)} > c\sqrt{- \log \lambda}\|dz\|_{L^2(\T)} \]
 when $\lambda$ is small enough. 
	
	For $df$, we demonstrate a similar behavior. Let $h_m(r)=h_{0m}(r)$ and $f(r, 0, z)= h_{m}(r)\cos(\frac{2\pi m}{\lambda}z)$.  We have
	\begin{align*}
		\|df\|_{L^2(\T)}^2 &= \int_{\partial \T_R} f*df = \pi \lambda\sinh R \cosh R \cdot h_m(R)h_m'(R) \\
		\tn{by } \eqref{h'(r)< h(r)8d^2 logcoshr /sinh^2r} & < (\frac{2\pi m}{\lambda})^2 \cdot  \pi h_m^2(R)\cdot  \lambda  \log \cosh R.	\numberthis \label{dfL2-tube}
		\end{align*} 
	As $R\rightarrow \infty$, $h_m(R) < \infty$ by [\citenum{specialFunctionsgrr}, Theorem 2.1.2].  On the other hand, 
	\begin{align*}
		\|df\|_{\li(\T)} & \geq \max\{|f_rdr|, |f_zdz|\}\\
		&= \max\limits_{r\in [0, R]}\{h_m'(r), \frac{2\pi m}{\lambda}\frac{h_m(r)}{\cosh r} \}\geq  \frac{2\pi m}{\lambda}\frac{h_m(0)}{\cosh 0}=  \frac{2\pi m}{\lambda}.
	\end{align*}  
	This implies that for any fixed $c>0$, for all $\lambda$ small enough, 
	$$	\|df\|_{\li(\T)} \geq \frac{2\pi m}{\lambda} > \frac{2\pi m}{\lambda}\cdot c\cdot \sqrt{\pi}h_m(R) \sqrt{\lambda} \log\cosh R = c\sqrt{-\log\lambda}\|df\|_{L^2}. $$
\end{proof}
This shows that the local form \eqref{local conjecture} of the Brock-Dunfield conjecture \eqref{eq l2 less than root inj thurston} does not hold. 

\bibliographystyle{halpha}
\bibliography{sample}
	
\end{document}